\documentclass[a4paper]{article}

\usepackage{etex}

\usepackage{amssymb}
\usepackage{amsmath}
\usepackage{mathtools}
\usepackage[amsmath,thmmarks,hyperref]{ntheorem}
\usepackage{enumerate}							% für Aufzaehlungen um Anstriche festlegen zu koennen
\usepackage{dsfont}								% fuer R reelle Zahlen und so
\usepackage{color}

\usepackage{tikz}
\usepackage{pgfplots}
\usepackage{booktabs}	% verbesserter Tabellensatz

\usepackage{geometry}								% Seitenraender und co
\usepackage[final,notref]{showkeys}	% 
\usepackage[bookmarksnumbered=true,hidelinks]{hyperref}
\usepackage{caption}

\theoremnumbering{arabic}
\theoremstyle{break}
\RequirePackage{latexsym}
\theorembodyfont{\itshape}
\theoremsymbol{}
\theoremheaderfont{\normalfont\bfseries}
\theoremseparator{}
	\newtheorem{theorem}{Theorem}[section]
	\newtheorem{lemma}[theorem]{Lemma}

\theorembodyfont{\upshape}
\theoremsymbol{\ensuremath{\clubsuit}}
	\newtheorem{remark}[theorem]{Remark}
	
\theorembodyfont{\upshape}
\theoremsymbol{}
	\newtheorem{example}[theorem]{Example}
	
\theoremstyle{nonumberplain}
\theoremheaderfont{\normalfont\bfseries}
\theoremseparator{:}
\theorembodyfont{\normalfont}
\theoremsymbol{\ensuremath{\square}}
\RequirePackage{amssymb}

        \newtheorem{proof}{Proof}
		\qedsymbol{\ensuremath{\square}}
\numberwithin{equation}{section}			%Abschnittsweise Nummerierung der Label in Gleichungen

\newcommand{\ZBR}{\mathds{R}}		%reele Zahlen
\newcommand{\ZBN}{\mathds{N}}		%natuerliche Zahlen
\newcommand{\KAO}{\mathcal{O}}

\newcommand{\eps}{\varepsilon}
\newcommand{\phj}{\varphi}
\newcommand{\tnorm}[1]{\left|\!\!\;\left|\!\!\;\left| {#1} \right|\!\!\;\right|\!\!\;\right|}
\newcommand{\norm}[1]{\left\| {#1} \right\|}
\newcommand{\snorm}[1]{\left| {#1} \right|}

\title{FEM-analysis on graded meshes for turning point problems exhibiting an interior layer}
\author{Simon Becher\footnote{Institute of Numerical Mathematics, 
Technical University of Dresden, Dresden D-01062, Germany.
\mbox{e-mail:} Simon.Becher@tu-dresden.de}}
\date{} %\today

\begin{document}

	\maketitle
	
	\begin{abstract}
		We consider singularly perturbed boundary value problems with a simple interior
		turning point whose solutions exhibit an interior layer. These problems are discretised
		using higher order finite elements on layer-adapted graded meshes proposed by Liseikin.
		We prove $\epsilon$-uniform error estimates in the energy norm. Furthermore, for linear
		elements we are able to prove optimal order $\epsilon$-uniform convergence in the $L^2$-norm
		on these graded meshes.
	\end{abstract}
	
	
	
	\section{Introduction}
	
	We consider singularly perturbed boundary value problems of the type
	\begin{subequations}
	\label{prob:intLayer}
	\begin{equation}
		\begin{aligned}
			-\varepsilon u''(x) + a(x) u'(x) + c(x) u(x) &= f(x) \qquad \text{in } (-1,1), \\
			u(-1) = \nu_{-1}, \quad u(1) &= \nu_1,
		\end{aligned}
	\end{equation}
	where $0 < \varepsilon \ll 1$ is a small parameter and $a, c, f$ are sufficiently smooth with
	\begin{equation}
		a(x) = - (x-x_0) b(x), \qquad b(x) > 0, \qquad c(x) \geq 0, \qquad c(x_0) > 0
	\end{equation}
	\end{subequations}
	for a point $x_0 \in (-1,1)$. Thus, the solution of~\eqref{prob:intLayer} exhibits an
	interior layer of ``cusp''-type at the simple interior turning point $x_0$.
	
	In the literature (see i.e.~\cite{BHK84}, \cite[p.~95]{Lis01}, \cite[Lemma~2.3]{SS94}) the
	bounds for such interior layers are well known. We have
	\begin{gather}
		\label{ieq:innerLayerBounds}
		\left| u^{(i)}(x,\eps) \right| \leq C \left( 1 + \left(\eps^{1/2} + |x-x_0| \right)^{\lambda-i}\right)
	\end{gather}
	where the parameter $\lambda$ satisfies $0 < \lambda < \bar{\lambda} := c(x_0)/|a'(x_0)|$.
	The estimate also holds for $\lambda = \bar{\lambda}$, if $\bar{\lambda}$ is not an integer.
	Otherwise there is an additional logarithmic factor, see references cited above.
	For convenience	we assume $x_0 = 0$ in the following.
	
	In the last decades a multitude of numerical methods has been developed to solve singularly
	perturbed problems with turning points and interior layers. For a general review we refer
	to~\cite{SRP13}. Many authors have considered finite difference methods. A selection of
	possible schemes for problems of the form~\eqref{prob:intLayer} may be found in~\cite{Far88}
	and the references therein. Also some layer-adapted meshes have been proposed to handle
	interior layers of ``cusp''-type. As an example Liseikin~\cite{Lis01} proved the $\eps$-uniform first
	order convergence of an upwind finite difference method on special graded meshes. Moreover,
	Sun and	Stynes~\cite{SS94} studied finite elements on a piecewise uniform mesh.
	
	We shall also analyse the finite element method, but on the graded meshes proposed by
	Liseikin which are described by the mesh generating function
	\begin{gather*}
		\phj(\xi,\varepsilon) =
		\begin{cases}
			\left(\varepsilon^{\alpha/2}
				+ \xi \left[(1+\varepsilon^{1/2})^\alpha-\varepsilon^{\alpha/2}\right]\right)^{1/\alpha}
				- \varepsilon^{1/2}
				& \text{ for }  0 \leq \xi \leq 1, \\
			\varepsilon^{1/2} - \left(\varepsilon^{\alpha/2}
				- \xi \left[(1+\varepsilon^{1/2})^\alpha - \varepsilon^{\alpha/2} \right] \right)^{1/\alpha}
				& \text{ for } 0 \geq \xi \geq -1,
		\end{cases}
	\end{gather*}
	where $0< \alpha \leq \lambda$. In order to handle these meshes, we adapt some
	basic ideas from~\cite[pp.~243--244]{Lis01}. While the strategy of Sun and Stynes
	in~\cite[Section~5]{SS94} is restricted to linear finite elements, our approach is
	more general. Thus, we are able to treat finite elements of higher order as well.
	
	Under certain assumptions, we prove $\varepsilon$-uniform
	convergence in the energy norm of the form
	\begin{gather*}
		\tnorm{u- u_N}_{\varepsilon} \leq C N^{-k}
	\end{gather*}
	for finite elements of order $k$, where $C$ may depend on $\alpha$ and $k$, see
	Theorem~\ref{th:energyErrorLiseikin}. On the basis of a supercloseness result we also
	give an optimal error estimate in the $L^2$-norm of the form
	\begin{gather*}
		\norm{u-u_N} \leq C N^{-2}
	\end{gather*}
	for linear finite elements, see Theorem~\ref{th:L2ErrorLiseikin}.
	Numerical experiments confirm our theoretical results.
	
	Notation: In this paper $C$ denotes a generic constant independent of $\eps$
	and the number of mesh points. Furthermore, for an interval $I$ the usual Sobolev
	spaces $H^1(I)$, $H_0^1(I)$, and $L^2(I)$ are used. The spaces of continuous and
	$k$ times continuously differentiable functions on $I$ are written as $C(I)$
	and $C^k(I)$, respectively. Let $\left(\cdot,\cdot\right)_I$ denote the usual
	$L^2(I)$ inner product and $\norm{\cdot}_I$ the $L^2(I)$-norm. We will also
	use the supremum norm on $I$ given by $\norm{\cdot}_{\infty,I}$ and the
	semi-norm in $H^1(I)$ given by $\snorm{\cdot}_{1,I}$. If $I = (-1,1)$, the
	index $I$ in inner products, norms, and semi-norms will be omitted. Additionally,
	for all $v \in H^1((-1,1))$ we define a weighted energy norm by
	\begin{gather*}
		\tnorm{v}_\eps := \left( \eps \snorm{v}_1^2 + \norm{v}^2\right)^{1/2}.
	\end{gather*}
	Further notation will be introduced later at the beginning of the sections
	where it is needed.

\section{The graded meshes proposed by Liseikin}
\label{sec:LisMesh}	
	
	The basic idea of Liseikin is to find a transformation $\phj(\xi,\eps)$ that eliminates
	the singularities of the solution when it is studied with respect to $\xi$. In our case
	the approach can be condensed to the task to find $\phj: [0,1] \to [0,1]$ such that
	\begin{gather}
		\label{ieq:meshFunc}
		\phj' \left(\phj + \eps^{1/2}\right)^{\lambda - 1} \leq C,
		\qquad \qquad \phj(0)=0, \qquad \phj(1)=1.
	\end{gather}
	The outcome of this approach is the mesh generating function
	\begin{gather}
		\label{eq:liseikinMeshGenerator}
		\phj(\xi,\eps) =
		\begin{cases}
			\left(\eps^{\alpha/2} + \xi \left[(1+\eps^{1/2})^\alpha-\eps^{\alpha/2}\right]\right)^{1/\alpha} - \eps^{1/2}
				& \text{ for }  0 \leq \xi \leq 1, \\
			\eps^{1/2} - \left(\eps^{\alpha/2} - \xi \left[(1+\eps^{1/2})^\alpha - \eps^{\alpha/2} \right] \right)^{1/\alpha}
				& \text{ for } 0 \geq \xi \geq -1,
		\end{cases}
	\end{gather}
	where $0< \alpha \leq \lambda$. By construction we have $\phj(0,\eps) = 0$
	and $\phj(\pm 1,\eps)=\pm 1$. Note that Liseikin derived the same transformation indirectly.
	Based on the principle of equidistribution, he used basic majorants of the solution
	derivatives to find basic layer-damping transformations. This procedure allows to handle
	also various other types of singularities, see e.g.~\cite[Chapter~6]{Lis01}.
	
	Now, the mesh points are generated by $x_i = \phj(\frac{i}{N},\eps)$, $i=-N,\ldots,N$.
	We will denote the lengths of the mesh intervals by $h_i := x_i - x_{i-1}$,
	$i = -N +1 ,\ldots, N$ and $h := N^{-1}$. Additionally, set $\hbar_i := (h_i+h_{i+1})/2$
	for $i = -N+1,\ldots,N-1$.

	Motivated by Lemma~\ref{le:1pcmc} and~\ref{le:1pcmcII}, we define and estimate a special
	constant $\kappa$ dependent on $\alpha \in (0,1]$ and $\eps \in (0,1]$ by
	\begin{gather}
		\label{kappa}
		 0< \ln(2) \leq \kappa
		 	:= \kappa(\alpha,\eps)
		 	:= \frac{(1+\eps^{1/2})^\alpha-\eps^{\alpha/2}}{\alpha}
		 	\leq \min\big\{\alpha^{-1},1 + \big|\log_2 (\eps^{1/2})\big|\big\}.
	\end{gather}
	It remains to check whether or not $\phj$ defined in~\eqref{eq:liseikinMeshGenerator}
	satisfies~\eqref{ieq:meshFunc}. An easy calculation shows
	\begin{align*}
		\frac{\partial \phj}{\partial \xi} \left( \phj + \eps^{1/2}\right)^{\lambda - 1}
			& = \frac{1}{\alpha} \left[(1+\eps^{1/2})^\alpha-\eps^{\alpha/2}\right]
				\left(\eps^{\alpha/2} + \xi \left[(1+\eps^{1/2})^\alpha-\eps^{\alpha/2}\right]\right)^{(1-\alpha)/\alpha}\\
			& \qquad \qquad
				\left(\eps^{\alpha/2} + \xi \left[(1+\eps^{1/2})^\alpha-\eps^{\alpha/2}\right]\right)^{(\lambda-1)/\alpha} \\
			& = \kappa \left(\eps^{\alpha/2} + \xi \left[(1+\eps^{1/2})^\alpha-\eps^{\alpha/2}\right]\right)^{(\lambda-\alpha)/\alpha}
				\leq C \kappa
	\end{align*}
	which can be bounded independent of $\eps$ due to~\eqref{kappa}.
	
	Since the arguments are very similar thanks to the symmetry of the mesh, we will consider the case
	$\xi \geq 0$ only. The next lemmas comprise some basic results concerning the mesh points and
	mesh intervals. Their proofs are deferred to Appendix~\ref{app:proofLiseikin}. The argumentation
	substantially uses the property~\eqref{ieq:meshFunc}. Here, the derivative of $\phj$ comes into play
	since the mean value theorem guarantees the estimate $h_i \leq h \frac{\partial \phj}{\partial \xi}(\xi_i,\eps)$
	for a $\xi_i \in (x_{i-1},x_i)$.
	\begin{lemma}
		\label{le:liseikinh^kBounds}
		Let $\hat{\alpha} > 0$ and $0< \alpha \leq \min\{\hat{\alpha}/k,1\}$ with $k \in \ZBN,\, k\geq 1$ then
		\begin{gather*}
			h_i^k \left( x_{i-1}+\eps^{1/2} \right)^{\hat{\alpha}-k} \leq 
			\begin{cases}
				C h^k		& \text{ for } 2 \leq i \leq N, \\
				h_1^k \eps^{(\hat{\alpha}-k)/2} \leq C h^k	& \text{ for } i = 1, \quad \eps \geq h^{2/\alpha}.
			\end{cases}
%			\qquad \qquad \left(C < 2^{k/\alpha} \kappa^k \right)
		\end{gather*}
		If $0 < \alpha \leq 1/k$ with $k \in \ZBN,\, k\geq 1$ and $\eps \leq h^{2/\alpha}$ then we have
		\begin{gather*}
			x_1 \leq C h^k.
%				\qquad \qquad \left(C \leq \big(2 h^{(1-k\alpha)}\big)^{1/\alpha} \leq 2^{1/\alpha} \; \text{under weak cond.} \; C \leq 1 \right)
		\end{gather*}
		In general, we have for $0 < \alpha \leq 1$
		\begin{gather*}
			h_i \leq C h
				\qquad \text{ for } 1 \leq i \leq N.
%				\qquad \qquad \left(C \leq 2 \kappa \leq 2 \min\big\{\alpha^{-1},1 + \big|\log_2 (\eps^{1/2})\big|\big\}\right)
		\end{gather*}
	\end{lemma}
	
	\begin{lemma}
		\label{le:liseikinhBounds}
		For $0 < \alpha \leq \frac{1}{2}$ the following inequality holds
		\begin{gather*}
			h_i - h_{i-1}
				\leq C h^2 \left( x_i + \eps^{1/2} \right)^{1-2\alpha}
				\qquad \text{ for } 2 \leq i \leq N.
		\end{gather*}
		Let $\hat{\alpha} > 0$ and $0 < \alpha \leq \min\{\hat{\alpha}/2,1/2\}$	then
		\begin{gather*}
			\left(h_i - h_{i-1}\right) \left( x_{i-1} + \eps^{1/2} \right)^{\hat{\alpha}-1}
				\leq C h^2
				\qquad \text{ for } 2 \leq i \leq N.
		\end{gather*}
%		where $C$ denotes a generic constant that is independent of $\eps$ and $h$ but may depends on $\alpha$.
	\end{lemma}
	
	\begin{remark}
		Note that an estimate similar to the first one of Lemma~\ref{le:liseikinhBounds}
		can also be found in~\cite{SunX92}.
	\end{remark}
	
	\begin{remark}
		In the FEM-analysis a generalised version of~\eqref{ieq:meshFunc}, i.e.
		\begin{gather*}
			\left(\phj'\right)^k \left(\phj + \eps^{1/2}\right)^{\lambda - k} \leq C,
			\qquad \qquad \phj(0)=0, \qquad \phj(1)=1,
		\end{gather*}
		would be convenient. In fact, this is ensured for $0 < \alpha \leq \lambda/k$
		which is already used in the proof of Lemma~\ref{le:liseikinh^kBounds}.
	\end{remark}
	
\section{FEM-analysis on graded meshes}
\label{sec:FEM-analysis}
	
	This section follows the paper of Sun and Stynes~\cite{SS94}, but while they studied
	linear finite elements on a layer-adapted piecewise uniform mesh, we shall use the
	graded mesh proposed by Liseikin instead. Besides our more general approach enables
	to analyse finite elements of higher order as well. We will only consider homogeneous
	Dirichlet boundary condition $\nu_{-1} = \nu_1 = 0$. This is no restriction at all
	since it can be easily ensured by modifying the right hand side $f$. Furthermore, due
	to~\cite[Lemma 2.1]{SS94} we may assume without loss of generality that
	\begin{gather}
		\label{gamma}
		\left( c - \tfrac{1}{2}a'\right)(x) \geq \gamma > 0 \qquad \text{ for all } x \in [-1,1], \quad
			\eps \text{ sufficiently small.}
	\end{gather}
	
	For $v, w \in H_0^1((-1,1))$ we set
	\begin{gather*}
		B_\eps\!\left(v,w\right) := \left(\eps v',w'\right) + \left(a v', w\right) + \left(c v,w\right)\!.
	\end{gather*}
	Note that the bilinear form $B_\eps\!\left(\cdot,\cdot\right)$ is uniformly coercive over
	$H_0^1((-1,1)) \times H_0^1((-1,1))$ in terms of the energy norm $\tnorm{\cdot}_\eps$
	thanks to~\eqref{gamma}.
	
	The weak formulation of~\eqref{prob:intLayer} with $\nu_{-1} = \nu_1 = 0$ reads
	as follows: \medskip
	
	Find $u \in H_0^1((-1,1))$ such that
	\begin{gather*}
		B_\eps\!\left(u,v\right) = \left(f,v \right), \qquad \text{ for all } v \in H_0^1((-1,1)).
	\end{gather*}
	
	Let $k \geq 1$ and let $P_k((x_a,x_b))$ denote the space of polynomial functions of maximal
	order $k$ over $(x_a,x_b)$. We define the trial and test space $V^N$ by
	\begin{gather*}
		V^N := \left\{ v \in C([-1,1]) : v|_{(x_{i-1},x_i)} \in P_k((x_{i-1},x_i))\, \forall i, \, v(-1) = v(1) = 0 \right\}.
	\end{gather*}
	Then the discrete problem is given by: \medskip
	
	Find $u_N \in V^N$ such that
	\begin{gather}
		\label{dprob:intLayer}
		B_\eps\!\left(u_N,v_N\right) = \left(f,v_N \right), \qquad \text{ for all } v_N \in V^N.
	\end{gather}
	
	Let $\hat{\phi}_0, \ldots, \hat{\phi}_k$ denote the Lagrange basis functions on the reference
	interval $[0,1]$ with respect to the points $0=\hat{x}_0 < \hat{x}_1 < \ldots < \hat{x}_k = 1$.
	We shall denote by $u_I \in V^N$ the interpolant of $u$ which is defined on each mesh interval
	$(x_{i-1},x_i)$ by
	\begin{gather*}
		u_I \big|_{(x_{i-1},x_i)} = \sum_{j=0}^k u(x_{i,j}) \phi_{i,j},
	\end{gather*}
	with $x_{i,j} := x_{i-1} + h_i \hat{x}_j$ and $\phi_{i,j}(x) := \hat{\phi}_j((x-x_{i-1})/h_i)$.
	
	Assuming $u \in C^{k+1}([x_{i-1},x_i])$, for all $j = 0, \ldots, k+1$ the standard interpolation theory
	leads to the error estimates: \medskip
	
	For $x \in (x_{i-1},x_i)$ there are $\xi_i^j \in (x_{i-1},x_i)$ such that
	\begin{gather}
		\left| (u-u_I)^{(j)}(x) \right|
			\leq C h_i^{k+1-j} \left| u^{(k+1)}(\xi_i^j) \right| \label{standInt} \\
		\intertext{and}
		\bigl\| (u-u_I)^{(j)} \bigr\|_{\infty,(x_{i-1},x_i)}
			\leq C \bigl\| u^{(j)} \bigr\|_{\infty,(x_{i-1},x_i)}. \label{standIntInfty}
	\end{gather}
	Furthermore, for all $v_N \in V^N$ the inverse inequality
	\begin{gather*}
		\snorm{v_N}_{1,(x_{i-1},x_i)} \leq C h_i^{-1} \norm{v_N}_{(x_{i-1},x_i)}
	\end{gather*}
	holds.
	
	\subsection{Finite elements of higher order}
	
	In the following we shall present the analysis for finite elements of order $k \geq 1$
	for problems of the form~\eqref{prob:intLayer}. We assume that $\lambda \in (0,k+1)$ which
	is the most difficult case. Otherwise all crucial derivatives of the solution could be
	bounded by a generic constant independent of $\eps$ and consequently optimal order
	$\eps$-uniform estimates could be proven with standard methods on uniform meshes.
	\begin{lemma}
		\label{le:energyuIuN_Pk}
		Let $u$ be the solution of~\eqref{prob:intLayer} and $u_N$ the solution of~\eqref{dprob:intLayer}
		on an arbitrary mesh. Then we have
		\begin{gather*}
			\tnorm{u_I - u_N}_\eps \leq C \tnorm{u_I-u}_\epsilon
				+ C \left(\sum_{i=-N+1}^N h_i^{-2} \norm{x (u_I-u)}_{(x_{i-1},x_i)}^2\right)^{\!1/2}.
		\end{gather*}
	\end{lemma}
	\begin{proof}
		By the coercivity of $B_\eps\!\left(\cdot,\cdot\right)$ and due to orthogonality, we have
		\begin{gather}
			\label{ieq:coerBeps}
			C \tnorm{u_I-u_N}^2_\eps \leq B_\eps(u_I-u_N,u_I-u_N) = B_\eps(u_I-u,u_I-u_N).
		\end{gather}
		Integrating by parts, we obtain
		\begin{multline*}
			B_\eps(u_I-u,u_I-u_N) \\
			\begin{aligned}
				& = \eps \bigl( (u_I-u)',(u_I-u_N)'\bigr)
					+ \bigl(a(u_I-u)',u_I-u_N\bigr) + \bigl(c(u_I-u),u_I-u_N\bigr) \\
				& = \eps \bigl( (u_I-u)',(u_I-u_N)'\bigr)
					- \bigl(a(u_I-u),(u_I-u_N)'\bigr) + \bigl((c-a')(u_I-u),u_I-u_N\bigr).
			\end{aligned}
		\end{multline*}
		Hence, triangle inequality and Cauchy-Schwarz inequality yield
		\begin{align}
			\left|B_\eps(u_I-u,u_I-u_N)\right| 
				& \leq \eps \left|\bigl( (u_I-u)',(u_I-u_N)'\bigr)\right|
					+ \left|\bigl((c-a')(u_I-u),u_I-u_N\bigr)\right| \nonumber \\
				& \qquad + \sum_{i=-N+1}^N\left|\bigl(a(u_I-u),(u_I-u_N)'\bigr)_{(x_{i-1},x_i)}\right| \nonumber \\
				& \leq \sqrt{\eps} \snorm{u_I-u}_1 \sqrt{\eps}\snorm{u_I-u_N}_1
					+ \norm{c-a'}_\infty \norm{u_I-u} \norm{u_I-u_N} \nonumber \\
				& \qquad + \sum_{i=-N+1}^N \norm{a(u_I-u)}_{(x_{i-1},x_i)} \snorm{u_I-u_N}_{1,(x_{i-1},x_i)}. \label{eq:sumi}
		\end{align}
		Now, for $-N+1 \leq i \leq N$ an inverse inequality and the fact that $a$ is smooth with $a(0)=0$ imply
		%\marginpar{$a\in C^{0,1}$\\needed}
		\begin{gather*}
			\norm{a(u_I-u)}_{(x_{i-1},x_i)} \snorm{u_I-u_N}_{1,(x_{i-1},x_i)}
				\leq C h_i^{-1} \norm{x(u_I-u)}_{(x_{i-1},x_i)} \norm{u_I-u_N}_{(x_{i-1},x_i)}.
		\end{gather*}
		Using this bound to estimate~\eqref{eq:sumi}, we get by Cauchy-Schwarz' inequality
		\begin{align*}
			\left|B_\eps(u_I-u,u_I-u_N)\right| 
				& \leq \max\bigl\{1\, ,\, \norm{c-a'}_\infty\bigr\}
					\bigl( \sqrt{\eps}\snorm{u_I-u}_1 + \norm{u_I-u}\bigr)\tnorm{u_I-u_N}_\eps \\
				& \qquad + \sum\nolimits_{i} C h_i^{-1} \norm{x(u_I-u)}_{(x_{i-1},x_i)} \norm{u_I-u_N}_{(x_{i-1},x_i)} \\
				& \leq \sqrt{2}\max\bigl\{1\, ,\, \norm{c-a'}_\infty\bigr\}
					\tnorm{u_I-u}_\eps \tnorm{u_I-u_N}_\eps \\
				& \qquad + C \left(\sum\nolimits_{i} h_i^{-2} \norm{x(u_I-u)}_{(x_{i-1},x_i)}^2\right)^{\!1/2}
					\left(\sum\nolimits_{i}\norm{u_I-u_N}_{(x_{i-1},x_i)}^2\right)^{\!1/2} \\
				& \leq \left(C \tnorm{u_I-u}_\eps 
					+ C \left(\sum\nolimits_{i} h_i^{-2} \norm{x(u_I-u)}_{(x_{i-1},x_i)}^2\right)^{\!1/2} \right)
					\tnorm{u_I-u_N}_\eps.
		\end{align*}
		Combining this and~\eqref{ieq:coerBeps} completes the proof.
	\end{proof}
	
	\begin{remark}
		For linear finite elements Sun and Stynes \cite[Lemma~5.2]{SS94} proved an estimate of the form
		\begin{gather*}
			\tnorm{u_I-u_N}_\eps \leq C \left(\norm{u-u_I}^{1/2} + \max_i{h_i^2}\right)\!,
		\end{gather*}
		see also Lemma~\ref{le:energyuIuNToL2uILiseikin}. Aside from the fact that their argumentation
		works for linear elements only, such an estimate would not enable optimal estimates for finite
		elements of higher order.
	\end{remark}
	
%	\begin{remark}
%		In the same setting as in Lemma~\ref{le:energyuIuN_Pk} one has
%		\begin{gather*}
%			\tnorm{u_I - u_N}_\eps \leq C \Big(\tnorm{u_I-u}_\eps + \norm{x (u_I-u)'}\Big).
%		\end{gather*}
%%		But, the effort that could be saved here has to be invested to bound $\norm{x(u_I-u)'}$,
%		Also note Remark~\ref{rem:xTimesInterDer}.
%	\end{remark}

	The next two lemmas give bounds for the interpolation error on the layer-adapted mesh proposed
	by Liseikin.
	\begin{lemma}
		\label{le:IntErrorLiseikin}
		Let $u$ be the solution of problem~\eqref{prob:intLayer}. Let $u_I \in V^N$ interpolate
		to $u$ on the mesh generated by~\eqref{eq:liseikinMeshGenerator} with
		$0< \alpha \leq \min\{\lambda/(k+1), 1/(2(k+1))\}$. Then
		\begin{gather}
			\label{ieq:L2IntLiseikin}
			\norm{u-u_I} \leq C N^{-(k+1)}
		\end{gather}
		and
		\begin{gather}
			\label{ieq:energyIntLiseikin}
			\tnorm{u-u_I}_\eps \leq C N^{-k}.
		\end{gather}
	\end{lemma}
	\begin{proof}
		Thanks to the symmetry of the problem, we shall consider only $x \in [0,1]$. Furthermore,
		we use $j \in \{0,1\}$ to switch between the $L^2$-norm term and the $\eps$-weighted
		$H^1$-seminorm term.
		
		Let $x \in (x_{i-1},x_i)$ where $2\leq i \leq N$. Then for some $\xi_i^j \in (x_{i-1},x_i)$
		\begin{align*}
			\eps^{j/2} \left| (u-u_I)^{(j)}(x) \right|
				& \leq C \eps^{j/2} h_i^{k+1-j} \left| u^{(k+1)}(\xi_i^j) \right| \\
				& \leq C \eps^{j/2} h_i^{k+1-j} \left(1 + \left( x_{i-1} + \eps^{1/2} \right)^{\lambda - (k+1)}\right) \\
				& \leq C N^{-(k+1-j)},
		\end{align*}
		where we used~\eqref{standInt}, \eqref{ieq:innerLayerBounds}, and Lemma~\ref{le:liseikinh^kBounds}. Hence,
		\begin{gather*}
			\eps^j \int_{x_1}^1 \left(( u-u_I)^{(j)}(x)\right)^2 dx \leq C N^{-2(k+1-j)}.
		\end{gather*}
		
		Now, let $x \in (x_0,x_1)$. We consider two different cases.
		
		First, if $\eps \geq h^{2/\alpha}$ then
		as above Lemma~\ref{le:liseikinh^kBounds} yields
		\begin{align*}
			\eps^{j/2}\left| (u-u_I)^{(j)}(x) \right|
				& \leq C \eps^{j/2} h_1^{k+1-j} \left(1 + \left( x_0 + \eps^{1/2} \right)^{\lambda - (k+1)}\right) \\
				& \leq C h_1^{k+1-j} \left(1+ \eps^{(\lambda-(k+1-j))/2}\right) \\
				& \leq C N^{-(k+1-j)}
		\end{align*}
		and therefore
		\begin{gather*}
			\eps^j \int_0^{x_1} \left( (u-u_I)^{(j)}(x)\right)^2 dx
				\leq \left( C N^{-(k+1-j)} \right)^2 \int_{0}^{x_1} \! 1 dx
				\leq C h_1 N^{-2(k+1-j)}.
		\end{gather*}
		
		If $\eps \leq h^{2/\alpha}$ we estimate the integral directly. We have
		\begin{gather*}
			\eps^j \int_0^{x_1} \left( (u-u_I)^{(j)}(x)\right)^2 dx
				\leq C \eps^j \int_0^{x_1} \bigl\|u^{(j)}\bigr\|_{\infty,(0,x_1)}^2 dx
				\leq C \eps^j \left(1+\eps^{(\lambda - j)/2}\right)^{\!2} x_1
				\leq C N^{-2(k+1)}
		\end{gather*}
		by~\eqref{standIntInfty}, \eqref{ieq:innerLayerBounds}, and Lemma~\ref{le:liseikinh^kBounds}.
		
		Combining the above estimates for $j=0$ and using symmetry on $[-1,0]$ we get~\eqref{ieq:L2IntLiseikin}.
		This estimate together with the above estimates for $j=1$ immediately gives~\eqref{ieq:energyIntLiseikin}.
	\end{proof}
	
	It remains to estimate the second term in Lemma~\ref{le:energyuIuN_Pk}.
	\begin{lemma}
		\label{le:sumIntLiseikin}
		Let $u$ be the solution of problem~\eqref{prob:intLayer}. Let $u_I \in V^N$ interpolate
		to $u$ on the mesh generated by~\eqref{eq:liseikinMeshGenerator} with
		$0< \alpha \leq \min\{\lambda/(k+1), 1/(2(k+1))\}$. Then
		\begin{gather}
			\label{ieq:sumIntLiseikin}
			\left(\sum_{i=-N+1}^N h_i^{-2} \norm{x (u_I-u)}_{(x_{i-1},x_i)}^2\right)^{\!1/2} \leq C N^{-k}.
		\end{gather}
	\end{lemma}
	\begin{proof}
		The proof is similar to the proof of Lemma~\ref{le:IntErrorLiseikin} but advanced in some way.
		
		Let $2\leq i \leq N$ and $x \in (x_{i-1},x_i)$. Then for some $\xi_i \in (x_{i-1},x_i)$
		\begin{align*}
			x \left| (u_I-u)(x) \right|
				& \leq C x h_i^{k+1} \left| u^{(k+1)}(\xi_i) \right| \\
				& \leq C (x_{i-1}+h_i) h_i^{k+1} \left(1 + \left( x_{i-1} + \eps^{1/2} \right)^{\lambda - (k+1)}\right) \\
				& \leq C h_i \left(N^{-k}+N^{-(k+1)}\right) \leq C h_i N^{-k},
		\end{align*}
		where we used~\eqref{standInt}, \eqref{ieq:innerLayerBounds}, and Lemma~\ref{le:liseikinh^kBounds}.
		Hence, for $2 \leq i \leq N$
		\begin{gather*}
			h_i^{-2} \norm{x(u_I-u)}_{(x_{i-1},x_i)}^2
				= h_i^{-2} \int_{x_{i-1}}^{x_i} \left( x( u_I-u)(x)\right)^2 dx
				\leq h_i^{-2} \left( C h_i N^{-k}\right)^2 \int_{x_{i-1}}^{x_i} \! 1 dx
				\leq C h_i N^{-2k}.
		\end{gather*}
		
		Now, let $i=1$ and $x \in (x_0,x_1)$. We consider two different cases.
		
		First, if $\eps \geq h^{2/\alpha}$ then
		as above Lemma~\ref{le:liseikinh^kBounds} yields
		\begin{align*}
			x \left| (u_I-u)(x) \right|
				& \leq C x h_1^{k+1} \left(1 + \left( x_0 + \eps^{1/2} \right)^{\lambda - (k+1)}\right) \\
				& \leq C h_1^{k+2} \left(1+ \eps^{(\lambda-(k+1))/2}\right) \\
				& \leq C h_1 N^{-(k+1)}
		\end{align*}
		and therefore
		\begin{gather*}
			h_1^{-2} \norm{x(u_I-u)}_{(x_0,x_1)}^2
				\leq h_1^{-2} \left( C h_1 N^{-(k+1)}\right)^2 \int_{x_0}^{x_1} \! 1 dx
				\leq C h_1 N^{-2(k+1)}.
		\end{gather*}
		
		If $\eps \leq h^{2/\alpha}$ we estimate the integral directly. We have
		\begin{align*}
			h_1^{-2} \int_0^{x_1} \left( x (u_I-u)(x)\right)^2 dx
				& \leq h_1^{-2} \norm{u_I-u}_{\infty,(0,x_1)}^2 \int_0^{x_1} x^2 dx \\
				& \leq C h_1^{-2} \norm{u}_{\infty,(0,x_1)}^2 x_1^3 \\
				& \leq C x_1 \leq C N^{-2(k+1)}
		\end{align*}
		by~\eqref{standIntInfty}, \eqref{ieq:innerLayerBounds}, and Lemma~\ref{le:liseikinh^kBounds}.
		
		Summing up the above estimates gives in the worst case
		\begin{gather*}
			\sum_{i=1}^N h_i^{-2} \norm{x(u_I-u)}_{(x_{i-1},x_i)}^2
				\leq C N^{-2(k+1)} + \sum_{i=2}^N C h_i N^{-2k}
				\leq C N^{-2k} \left(\sum_{i=2}^N h_i + N^{-2}\right)
				\leq C N^{-2k}.
		\end{gather*}
		Since, thanks to symmetry the sum for $i=-N+1,\ldots,0$ can be bounded analogously, the proof
		is completed.
	\end{proof}
	
%	\begin{remark}
%		\label{rem:xTimesInterDer}
%		Similar to Lemma~\ref{le:sumIntLiseikin} one can prove
%		\begin{gather*}
%			\norm{x (u_I-u)'} \leq C N^{-k}.
%		\end{gather*}
%		However, the arguments that allow to estimate the integral over $(0,x_1)$
%		directly for $\eps \leq h^{2/\alpha}$ are somewhat more involved. For linear
%		elements this can be circumvented, see Lemma~\ref{le:energyuIuNToL2uI}.
%	\end{remark}
	
	Now, we are able to prove the $\eps$-uniform error estimate of $P_k$-FEM in the energy norm.
	\begin{theorem}
		\label{th:energyErrorLiseikin}
		Let $u$ be the solution of~\eqref{prob:intLayer} and $u_N$ the solution
		of~\eqref{dprob:intLayer} on a mesh generated by~\eqref{eq:liseikinMeshGenerator} with
		$0 < \alpha \leq \min\{\lambda/(k+1), 1/(2(k+1))\}$. Then we have
		\begin{gather*}
			\tnorm{u-u_N}_\eps \leq C N^{-k}.
		\end{gather*}
%		where the generic constant $C$ is independent of $\eps$ and $N$, but may depend on $\alpha$ and $k$.
	\end{theorem}
	\begin{proof}
		The bound in the energy norm follows easily from the triangle inequality, Lemma~\ref{le:energyuIuN_Pk},
		\eqref{ieq:energyIntLiseikin}, and~\eqref{ieq:sumIntLiseikin}.
	\end{proof}
	
%	\begin{remark}
%		A similar FEM-analysis is possible on a slightly modified version of the piecewise
%		equidistant mesh proposed by Sun and Stynes~\cite{SS94}. The strategies can also
%		be used to analyse the streamline diffusion finite element method. This will be
%		part of further publications.
%	\end{remark}
	
\subsection{Special features of linear finite elements}
	In this section we present some special features of linear finite elements. So, we shall
	assume $k=1$. The following two lemmas hold for arbitrary meshes and are borrowed
	from~\cite{SS94}. In particular, they show that for linear finite elements the $L^2$-norm
	interpolation error estimate~\eqref{ieq:L2IntLiseikin} suffices to prove the $\eps$-uniform
	convergence in the energy norm.
	\begin{lemma}[see {\cite[Lemma~5.1]{SS94}}]
		\label{le:energyuIuNToL2uI}
		Let $u$ be the solution of problem~\eqref{prob:intLayer} and $k=1$.
		Then on an arbitrary mesh we have
		\begin{gather}
			\label{ieq:energyToL2IntLiseikin}
			\tnorm{u-u_I}_\eps^2 \leq C \norm{u-u_I}
		\end{gather}
		and
		\begin{gather}
			\label{ieq:x(uMuI)'}
			\int_{-1}^1 \left(x \left(u-u_I\right)'(x)\right)^2 dx \leq C \norm{u-u_I}\!.
		\end{gather}
	\end{lemma}
	\begin{proof}
		See \cite[Lemma~5.1]{SS94}.
	\end{proof}
	
	\begin{lemma}[similar to {\cite[Lemma~5.2]{SS94}}]
		\label{le:energyuIuNToL2uILiseikin}
		Let $u$ be the solution of~\eqref{prob:intLayer} and $u_N \in V^N$ $(k=1)$ the solution
		of~\eqref{dprob:intLayer} on an arbitrary mesh. Then we have
		\begin{gather*}
			\tnorm{u_I - u_N}_\eps \leq C \norm{u-u_I}^{1/2}.
		\end{gather*}
	\end{lemma}
	\begin{proof}
		As in the proof of Lemma~\ref{le:energyuIuN_Pk} the coercivity of $B_\eps(\cdot,\cdot)$
		and orthogonality yield
		\begin{gather*}
			C \tnorm{u_I - u_N}_\eps^2 \leq B_\eps\!\left(u_I - u_N, u_I -u_N\right)
				= B_\eps\!\left(u_I - u, u_I -u_N\right)\!.
		\end{gather*}
		Integrating by parts and applying the Cauchy-Schwarz inequality, we have
		\begin{align*}
			&\left|B_\eps\!\left(u_I - u, u_I -u_N\right)\right| \\
			& \qquad \leq | \eps ( u_I - u, \underbrace{(u_I - u_N)''}_{=0} ) |
					+ \left| \left(a (u_I - u)', u_I - u_N \right) \right|
					+ \left| \left(c (u_I - u), u_I - u_N \right) \right| \\
			& \qquad \leq C \norm{x (u_I - u)'} \norm{u_I - u_N}
					+ C \norm{u_I - u} \norm{u_I - u_N} \\
			& \qquad \leq C \norm{u_I - u}^{1/2} \norm{u_I - u_N}\!,
		\end{align*}
		where we used~\eqref{ieq:x(uMuI)'} and $\norm{u_I}_\infty \leq \norm{u}_\infty \leq C$.
	\end{proof}
	
	The next lemma provides an auxiliary inequality that will be needed later. We defer its
	proof to Appendix~\ref{app:proofNorm}.
	\begin{lemma}
		\label{le:normEquiVN}
		Let $e \in V^N$ $(k=1)$ on an arbitrary mesh and $-N\leq L < R \leq N$. Then
		\begin{gather*}
			\sum_{i=L+1}^{R-1} \hbar_i |e_i| + \frac{1}{2} \left(h_{L+1} |e_L| + h_{R}|e_R| \right)
				\leq C \norm{e}_{(x_L,x_R)}\!,
		\end{gather*}
		where $e_i = e(x_i)$.
	\end{lemma}
	\begin{proof}
		See Appendix~\ref{app:proofNorm}.
	\end{proof}
	
	\begin{lemma}[Supercloseness]
		\label{le:supercloseLiseikin}
		Let $u$ be the solution of~\eqref{prob:intLayer} and $u_N \in V^N$ $(k=1)$ the solution
		of~\eqref{dprob:intLayer} on a mesh generated by~\eqref{eq:liseikinMeshGenerator} with
		$0 < \alpha \leq \min\{\lambda/2,1/4\}$. Then we have
		\begin{gather*}
			\tnorm{u_I - u_N}_\eps \leq C N^{-2}.
		\end{gather*}
	\end{lemma}
	\begin{proof}
		Following the argument of Lemma~\ref{le:energyuIuNToL2uILiseikin} we have
		\begin{gather}
			\label{estLemma:first}
			C \tnorm{u_I-u_N}_\eps^2 \leq \left| \left( a (u_I-u)',u_I-u_N \right) \right|
				+ C \norm{u_I-u} \norm{u_I-u_N}\!.
		\end{gather}
		
		Integrating by parts and applying the Cauchy-Schwarz inequality yield for $j \in \{0,1\}$
		\begin{gather}
			\label{estLemma:axi1}
			\left| \left( a (u_I-u)',u_I-u_N \right)_{(x_j,1)} \right|
				\leq \left| \left( a (u_I-u),(u_I-u_N)' \right)_{(x_j,1)} \right|
					+ C \norm{u_I-u}_{(x_j,1)} \norm{u_I-u_N}_{(x_j,1)}\!.
		\end{gather}
		Set $e_i = \left( u_I - u_N \right)(x_i)$ for $i = -N,\ldots,N$. Then we have
		%e_N = u_I(1)-u_N(1) = u(1) - u_N(1) = 0
		\begin{gather}
			\label{estLemma:y1y2y3}
			\begin{aligned}
				& \left(a(u-u_I),(u_I-u_N)'\right)_{(x_1,1)} \\
				& \qquad = \sum_{i=2}^N \frac{e_i-e_{i-1}}{h_i} \int_{x_{i-1}}^{x_i} a(x) (u-u_I)(x) dx \\
				& \qquad = \sum_{i=2}^{N-1} e_i a(x_{i-1})
					\left\{\left( \frac{1}{h_i} \int_{x_{i-1}}^{x_i} - \frac{1}{h_{i+1}} \int_{x_i}^{x_{i+1}} \right)
						(u-u_I)(x)dx \right\} \\
				& \qquad \qquad + \sum_{i=2}^{N-1} e_i
					\left\{\left( \frac{1}{h_i} \int_{x_{i-1}}^{x_i} - \frac{1}{h_{i+1}} \int_{x_i}^{x_{i+1}} \right)
						\left(a(x)-a(x_{i-1})\right)(u-u_I)(x)dx \right\} \\
				& \qquad \qquad - e_1 \frac{1}{h_{2}} \int_{x_1}^{x_2} \left(a(x)-a(x_0)\right)(u-u_I)(x)dx \\
				& \qquad = Y_1 + Y_2 + Y_3. 
			\end{aligned}
		\end{gather}
		
		Inspecting the proof of Lemma~\ref{le:IntErrorLiseikin}, we see that
		\begin{gather*}
			\left|(u-u_I)(x)\right| \leq C N^{-2} \quad
			\begin{cases}
				\text{ if } x_1 \leq x \leq 1, \\
				\text{ if } x_1 > x \geq 0, \quad \eps \geq h^{2/\alpha}.
			\end{cases}
		\end{gather*}
		Consequently, we have for $i = 2,\ldots, N-1$ or if $i=1$ and $\eps \geq h^{2/\alpha}$
		\begin{gather*}
			\frac{1}{h_i} \int_{x_{i-1}}^{x_i} \left(a(x)-a(x_{i-1})\right) (u-u_I)(x) dx
				\leq C \hbar_i \norm{u-u_I}_{\infty,(x_{i-1},x_i)}
				\leq C \hbar_i N^{-2},
		\end{gather*}
		and analogously for $i = 1,\ldots,N-1$
		\begin{gather*}
			\frac{1}{h_{i+1}} \int_{x_i}^{x_{i+1}} \left(a(x)-a(x_{i-1})\right) (u-u_I)(x) dx
				\leq C \hbar_i \norm{u-u_I}_{\infty,(x_i,x_{i+1})}
				\leq C \hbar_i N^{-2}.
		\end{gather*}
		Hence, recalling Lemma~\ref{le:normEquiVN}, we conclude
		\begin{gather}
			\label{estLemma:y2y3}
			\left| Y_2 \right| + \left|Y_3\right|
				\leq C N^{-2} \sum_{i=1}^{N-1} e_i \hbar_i
				\leq C N^{-2} \norm{u_I-u_N}_{(0,1)}
		\end{gather}
		and if $\eps \geq h^{2/\alpha}$ (note that $x_1 = h_1$)
		\begin{gather}
			\label{estLemma:a0x1}
			\begin{aligned}
				& \left| (a(u_I-u),(u_I-u_N)')_{(0,x_1)} \right| 
					= \left| \frac{e_1-e_0}{h_1} \int_0^{x_1} a(x) (u-u_I)(x) dx \, \right| \\
				& \qquad \qquad \leq C \frac{|e_1| + |e_0|}{h_1} \norm{u-u_I}_{\infty,(0,x_1)} \int_0^{x_1} x \, dx 
					\leq C N^{-2} \norm{u_I-u_N}_{(0,x_1)}\!.
			\end{aligned}
		\end{gather}
		
		Next, we bound $\left|Y_1\right|$. By an integral transformation, standard interpolation
		error estimates, the mean value theorem, and~\eqref{ieq:innerLayerBounds}, we obtain
		\begin{align*}
			& \left| \left( \frac{1}{h_i}\int_{x_{i-1}}^{x_i} - \frac{1}{h_{i+1}}\int_{x_i}^{x_{i+1}} \right) (u-u_I)(x) \, dx \, \right| \\
			& \qquad = \left| \int_0^1 (u-u_I)(x_{i-1}+th_i) - (u-u_I)(x_i+th_{i+1}) \, dt \, \right| \\
			& \qquad = \frac{1}{2}\left| \int_0^1 (1-t)t \left[ h_i^2 u''(\xi_{i-1}(t)) - h_{i+1}^2 u''(\xi_{i+1}(t))\right] dt \, \right| \\
			& \qquad = \frac{1}{2}\left| \int_0^1 (1-t)t \left[ h_i^2 \left( u''(\xi_{i-1}(t)) - u''(\xi_{i+1}(t))\right)
				+ \left(h_i^2 - h_{i+1}^2\right) u''(\xi_{i+1}(t)) \right] dt \, \right| \\
			& \qquad \leq \frac{1}{2}\left( 2 \hbar_i h_i^2 \max_{x_{i-1}\leq \xi \leq x_{i+1}} \left| u'''(\xi) \right|
				+ 2 \hbar_i \left(h_{i+1} - h_i \right) \max_{x_i \leq \xi \leq x_{i+1}} \left| u''(\xi) \right| \right)
				\int_0^1 (1-t)t \, dt \\
			& \qquad = \frac{1}{6} \hbar_i \left(h_i^2 \max_{x_{i-1}\leq \xi \leq x_{i+1}} \left| u'''(\xi) \right|
				+ \left(h_{i+1} - h_i \right) \max_{x_i \leq \xi \leq x_{i+1}} \left| u''(\xi) \right| \right) \\
			& \qquad \leq C \hbar_i \left( h_i^2 + \left(h_{i+1} - h_i \right)
				+ h_i^2 \left( x_{i-1} + \eps^{1/2} \right)^{\lambda - 3}
				+ \left(h_{i+1} - h_i \right) \left( x_i + \eps^{1/2} \right)^{\lambda - 2} \right)\!,
		\end{align*}
		where $x_{i-1} < \xi_{i-1}(t) < x_i$ and $x_i < \xi_{i+1}(t) < x_{i+1}$.
		Combining this with the estimates of Lemma~\ref{le:liseikinh^kBounds}
		and Lemma~\ref{le:liseikinhBounds} yields together with Lemma~\ref{le:normEquiVN}
		\begin{gather}
			\label{estLemma:y1}
			\begin{aligned}
				\left| Y_1 \right|
					& \leq C \sum_{i=2}^{N-1} e_i x_{i-1}
						\left\{ \left| \left( \frac{1}{h_i} \int_{x_{i-1}}^{x_i} - \frac{1}{h_{i+1}} \int_{x_i}^{x_{i+1}} \right)
							(u-u_I)(x)dx \, \right| \right\} \\
					& \leq C N^{-2} \sum_{i=2}^{N-1} e_i x_{i-1} \hbar_i
						\left( 1 + \left( x_{i-1} + \eps^{1/2} \right)^{-1} + \left( x_i + \eps^{1/2} \right)^{-1} \right) \\
					& \leq C N^{-2} \sum_{i=2}^{N-1} e_i \hbar_i \leq C N^{-2} \norm{u_I - u_N}_{(x_1,1)}\!.
			\end{aligned}
		\end{gather}
		
		Altogether~\eqref{estLemma:axi1} -- \eqref{estLemma:y1},
		%\eqref{estLemma:axi1}, \eqref{estLemma:y1y2y3}, \eqref{estLemma:y2y3}, \eqref{estLemma:a0x1}, \eqref{estLemma:y1},
		and~\eqref{ieq:L2IntLiseikin} give
		\begin{gather*}
			\left| \left( a (u_I-u)',u_I-u_N \right)_{(x_j,1)} \right|
				\leq C N^{-2} \left(\norm{u_I - u_N}_{(x_j,1)} + \norm{u_I - u_N}_{(0,1)}\right)
		\end{gather*}
		for $j=1$ or $j=0$ if $\eps \geq h^{2/\alpha}$. Because of symmetry, it remains
		to bound $\left(a(u_I-u)',u_I-u_N\right)_{(0,x_1)}$ if $\eps \leq h^{2/\alpha}$.
		Applying the Cauchy-Schwarz inequality gives
		\begin{gather*}
			\left|\left(a(u_I-u)',u_I-u_N\right)_{(0,x_1)} \right|
				\leq C \left(\int_0^{x_1} \left(x(u_I-u)'\right)^2 dx\right)^{1/2} \norm{u_I-u_N}_{(0,x_1)}\!.
		\end{gather*}
		Using integration by parts twice, \eqref{ieq:innerLayerBounds}, and $\norm{u_I}_\infty \leq \norm{u}_\infty \leq C$, we obtain
		\begin{align*}
			\int_0^{x_1} \left(x(u_I-u)'\right)^2 dx
				& = - \int_0^{x_1} 2x \left(u_I - u\right)' \left( u_I - u \right) dx
					- \int_0^{x_1} x^2 u'' \left(u_I - u\right) dx \\
				& = \int_0^{x_1} 2x \left(u_I -u \right) \left( u_I-u\right)' dx
					+ \int_0^{x_1} 2 \left(u_I-u\right)^2 dx
					- \int_0^{x_1} x^2 u'' \left(u_I - u \right) dx \\
				& = \int_0^{x_1} \left(u_I-u\right)^2 dx
					- \int_0^{x_1} x^2 u'' \left(u_I - u \right) dx \\
				& \leq \int_0^{x_1} \left((u-u_I)^2(x) + C\left|(u-u_I)(x)\right| \right) dx \\
				& \leq C x_1 \leq C N^{-4}
		\end{align*}
		by Lemma~\ref{le:liseikinh^kBounds}.
		
		Combining~\eqref{estLemma:first}, the last three estimates, and~\eqref{ieq:L2IntLiseikin}
		completes the proof.
	\end{proof}
	
	We now prove the $\eps$-uniform error estimate in the energy and $L^2$-norm.
	\begin{theorem}
		\label{th:L2ErrorLiseikin}
		Let $u$ be the solution of~\eqref{prob:intLayer} and $u_N \in V^N$ $(k=1)$ the solution
		of~\eqref{dprob:intLayer} on a mesh generated by~\eqref{eq:liseikinMeshGenerator} with
		$0 < \alpha \leq \min\{\lambda/2,1/4\}$. Then we have
		\begin{gather*}
			\tnorm{u-u_N}_\eps \leq C N^{-1}
		\end{gather*}
		and
		\begin{gather*}
			\norm{u-u_N} \leq C N^{-2}.
		\end{gather*}
%		where the generic constant $C$ is independent of $\eps$ and $N$, but may depend on $\alpha$ and $k$.
	\end{theorem}
	\begin{proof}
		The bound in the energy norm is already given in Theorem~\ref{th:energyErrorLiseikin}, but
		also follows easily from the triangle inequality, Lemma~\ref{le:energyuIuNToL2uILiseikin},
		\eqref{ieq:energyToL2IntLiseikin}, and~\eqref{ieq:L2IntLiseikin}. To prove the bound in the
		$L^2$-norm only the supercloseness result of Lemma~\ref{le:supercloseLiseikin}
		and~\eqref{ieq:L2IntLiseikin} have to be used.
	\end{proof}
	
	\begin{remark}
		For linear elements in~\cite[Theorem~5.1]{SS94} the presumably non-optimal $L^2$-norm
		estimate
		\begin{gather*}
			\norm{u-u_N} \leq C\left(N^{-1}\ln N\right)^{3/2}
		\end{gather*}
		is proven for a discrete solution calculated on a piecewise-equidistant mesh. The
		argumentation there is similar to the proof of Lemma~\ref{le:supercloseLiseikin}.
		But, in one of the occurring terms there are problems when $h_i \neq h_{i+1}$ since
		the difference $|h_i-h_{i+1}|$ is not sufficiently small on the piecewise-equidistant
		mesh. For the graded meshes generated by~\eqref{eq:liseikinMeshGenerator} we have the
		estimates of Lemma~\ref{le:liseikinhBounds}. Thus, these problems can be
		circumvented.
	\end{remark}
	
\section{Numerical experiments}
\label{sec:numExp}
	
	Now, we shall present some numerical results to verify the theoretical
	findings of this paper. Therefore, we study a test problem taken from~\cite{SS94}
	whose solution exhibits typical interior layer behaviour of ``cusp''-type.
	
	All computations where performed using a FEM-code based on $\mathbb{SOFE}$ by Lars
	Ludwig~\cite{SOFE}. In general, the parameter $\alpha$ needed to generate the graded mesh was
	chosen as
	\begin{gather*}
		\alpha = \alpha_0 \min\{\lambda/(k+1), 1/(2(k+1))\}
	\end{gather*}
	with $\alpha_0=1$. Exceptions are explicitly stated. For given errors $E_{\eps,N}$
	we calculate the rates of convergence by $\left(\ln E_{\eps,N}-\ln E_{\eps,2N}\right)/\ln 2$.
	\begin{example}[see~\cite{SS94}]
		\label{test:SunStynes}
		We consider the singularly perturbed turning point problem
		\begin{align*}
			- \eps u'' - x(1+x^2) u' + \lambda(1+x^3)u & = f, \qquad \text{for} \quad x \in (-1,1), \\
			u(-1)=u(1) & = 0,
		\end{align*}
		where the right hand side $f(x)$ is chosen such that the solution $u(x)$ is given by
		\begin{gather*}
			u(x) = \left(x^2+\eps\right)^{\lambda/2} + x\left(x^2+\eps\right)^{(\lambda-1)/2}
				- \left(1+\eps\right)^{\lambda/2}\left(1+x\left(1+\eps\right)^{-1/2}\right).
		\end{gather*}
		Note that the parameter $\lambda$ in the problem coincide with the quantity
		$\bar{\lambda}=c(0)/|a'(0)|$.
	\end{example}
	
	In Figure~\ref{fig:energy} the energy norm error is plotted for finite elements of order
	$k = 1, \ldots, 4$ applied to Example~\ref{test:SunStynes} with
	$\eps = 10^{-8}$ and $\lambda = 0.005$. The expected convergence behaviour,
	cf. Theorem~\ref{th:energyErrorLiseikin}, can be clearly seen. The numerical
	results suggest that the energy norm error is almost independent of $\eps$. Anyway
	it stays stable for small $\eps$, see Table~\ref{tab:energy}.
	
	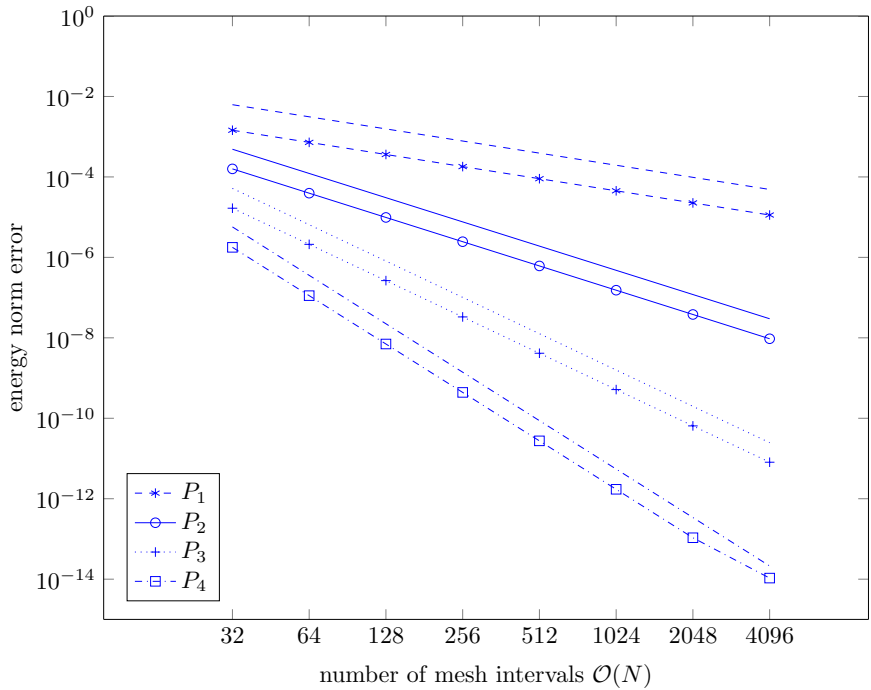
\begin{figure}
		\begin{center}
% energy norm error for the testproblem of Sun and Stynes
% with lambda = 0.005 and epsilon = 10^{-8}
%
%
\begin{tikzpicture}[scale=0.9]

\begin{axis}[%
width=4.4in,
height=3.5in,
at={(0.809in,0.513in)},
scale only axis,
separate axis lines,
every outer x axis line/.append style={black},
every x tick label/.append style={font=\color{black}},
xmode=log,
xmin=10,
xmax=10000,
xtick={  32,   64,  128,  256,  512, 1024, 2048, 4096},
xticklabels={  32,   64,  128,  256,  512, 1024, 2048, 4096},
xminorticks=true,
xlabel={number of mesh intervals $\KAO(N)$},
every outer y axis line/.append style={black},
every y tick label/.append style={font=\color{black}},
ymode=log,
ymin=1e-15,
ymax=1,
yminorticks=true,
ylabel={energy norm error},
axis background/.style={fill=white},
legend style={at={(0.03,0.03)},anchor=south west,legend cell align=left,align=left,draw=black}
]
\addplot [color=blue,dashed,mark=asterisk,mark options={solid}]
  table[row sep=crcr]{%
32	0.00144877361158874\\
64	0.000723347395832683\\
128	0.000361538252807036\\
256	0.000180751741895945\\
512	9.03736839481367e-05\\
1024	4.5186568165753e-05\\
2048	2.25932498433125e-05\\
4096	1.12966206412881e-05\\
};
\addlegendentry{$P_1$};

\addplot [color=blue,solid,mark=o,mark options={solid}]
  table[row sep=crcr]{%
32	0.000158527117878508\\
64	3.9436984428203e-05\\
128	9.8360947882653e-06\\
256	2.45473627405724e-06\\
512	6.12490802330366e-07\\
1024	1.52746537638214e-07\\
2048	3.80683908603278e-08\\
4096	9.4819339994192e-09\\
};
\addlegendentry{$P_2$};

\addplot [color=blue,dotted,mark=+,mark options={solid}]
  table[row sep=crcr]{%
32	1.66900355101386e-05\\
64	2.10908816362078e-06\\
128	2.64554315309594e-07\\
256	3.3096847810157e-08\\
512	4.13794786601101e-09\\
1024	5.17269423841341e-10\\
2048	6.46594803507615e-11\\
4096	8.08246350654861e-12\\
};
\addlegendentry{$P_3$};

\addplot [color=blue,dashdotted,mark=square,mark options={solid}]
  table[row sep=crcr]{%
32	1.77513454075289e-06\\
64	1.11897765271855e-07\\
128	7.02473735988224e-09\\
256	4.39451925317352e-10\\
512	2.74652900840514e-11\\
1024	1.71604482089942e-12\\
2048	1.07215193461717e-13\\
4096	1.06406549502993e-14\\
};
\addlegendentry{$P_4$};

\addplot [color=blue,dashed,forget plot]
  table[row sep=crcr]{%
32	0.00625\\
64	0.003125\\
128	0.0015625\\
256	0.00078125\\
512	0.000390625\\
1024	0.0001953125\\
2048	9.765625e-05\\
4096	4.8828125e-05\\
};
\addplot [color=blue,solid,forget plot]
  table[row sep=crcr]{%
32	0.00048828125\\
64	0.0001220703125\\
128	3.0517578125e-05\\
256	7.62939453125e-06\\
512	1.9073486328125e-06\\
1024	4.76837158203125e-07\\
2048	1.19209289550781e-07\\
4096	2.98023223876953e-08\\
};
\addplot [color=blue,dotted,forget plot]
  table[row sep=crcr]{%
32	5.18798828125e-05\\
64	6.4849853515625e-06\\
128	8.10623168945312e-07\\
256	1.01327896118164e-07\\
512	1.26659870147705e-08\\
1024	1.58324837684631e-09\\
2048	1.97906047105789e-10\\
4096	2.47382558882236e-11\\
};
\addplot [color=blue,dashdotted,forget plot]
  table[row sep=crcr]{%
32	5.7220458984375e-06\\
64	3.57627868652344e-07\\
128	2.23517417907715e-08\\
256	1.39698386192322e-09\\
512	8.73114913702011e-11\\
1024	5.45696821063757e-12\\
2048	3.41060513164848e-13\\
4096	2.1316282072803e-14\\
};
\end{axis}
\end{tikzpicture}%
			\caption{Energy norm error for finite elements of order $k = 1, \ldots, 4$ applied
			to Example~\ref{test:SunStynes} with $\eps = 10^{-8}$ and $\lambda = 0.005$.
			Reference curves of the form $\KAO(N^{-k})$.}
			\label{fig:energy}
		\end{center}
	\end{figure}
	
	\begin{table}
	\begin{center}
	\begin{tabular}{ l *{8}{c} } \toprule
			& \multicolumn{2}{c}{$P_1$-elements} & \multicolumn{2}{c}{$P_2$-elements} &
				\multicolumn{2}{c}{$P_3$-elements} & \multicolumn{2}{c}{$P_4$-elements} \\
		\cmidrule(l{0.7ex}r{0.7ex}){2-3} \cmidrule(l{0.7ex}r{0.7ex}){4-5}
		\cmidrule(l{0.7ex}r{0.7ex}){6-7} \cmidrule(l{0.7ex}r{0.7ex}){8-9}
		$\eps$ \textbf{\textbackslash} $N$ & 512 & 1024 & 512 & 1024 & 512 & 1024 & 512 & 1024 \\ 
		\cmidrule(l{0.7ex}r{0.7ex}){1-1} \cmidrule(l{0.7ex}r{0.7ex}){2-2} \cmidrule(l{0.7ex}r{0.7ex}){3-3}
		\cmidrule(l{0.7ex}r{0.7ex}){4-4} \cmidrule(l{0.7ex}r{0.7ex}){5-5} \cmidrule(l{0.7ex}r{0.7ex}){6-6}
		\cmidrule(l{0.7ex}r{0.7ex}){7-7} \cmidrule(l{0.7ex}r{0.7ex}){8-8} \cmidrule(l{0.7ex}r{0.7ex}){9-9} 
		$1$ & 5.89e-04 & 2.95e-04 & 2.36e-07 & 5.91e-08 & 1.37e-10 & 1.71e-11 & 1.49e-13 & 2.06e-13 \\ 
%		$10^{- 1}$ & 8.10e-04 & 4.05e-04 & 8.24e-07 & 2.06e-07 & 8.79e-10 & 1.10e-10 & 9.08e-13 & 3.66e-13 \\ 
		$10^{- 2}$ & 7.64e-04 & 3.82e-04 & 1.31e-06 & 3.28e-07 & 2.36e-09 & 2.95e-10 & 4.07e-12 & 3.06e-13 \\ 
%		$10^{- 3}$ & 6.21e-04 & 3.11e-04 & 1.55e-06 & 3.88e-07 & 4.06e-09 & 5.07e-10 & 1.02e-11 & 6.39e-13 \\ 
		$10^{- 4}$ & 4.61e-04 & 2.30e-04 & 1.52e-06 & 3.81e-07 & 5.28e-09 & 6.60e-10 & 1.75e-11 & 1.10e-12 \\ 
%		$10^{- 5}$ & 3.22e-04 & 1.61e-04 & 1.33e-06 & 3.33e-07 & 5.75e-09 & 7.19e-10 & 2.38e-11 & 1.49e-12 \\ 
		$10^{- 6}$ & 2.16e-04 & 1.08e-04 & 1.07e-06 & 2.68e-07 & 5.56e-09 & 6.95e-10 & 2.76e-11 & 1.73e-12 \\ 
%		$10^{- 7}$ & 1.41e-04 & 7.06e-05 & 8.22e-07 & 2.05e-07 & 4.95e-09 & 6.19e-10 & 2.87e-11 & 1.79e-12 \\ 
		$10^{- 8}$ & 9.04e-05 & 4.52e-05 & 6.12e-07 & 1.53e-07 & 4.14e-09 & 5.17e-10 & 2.75e-11 & 1.72e-12 \\ 
%		$10^{- 9}$ & 5.69e-05 & 2.85e-05 & 4.59e-07 & 1.14e-07 & 3.30e-09 & 4.13e-10 & 2.48e-11 & 1.55e-12 \\ 
		$10^{-10}$ & 3.54e-05 & 1.77e-05 & 3.69e-07 & 9.17e-08 & 2.54e-09 & 3.17e-10 & 2.17e-11 & 1.35e-12 \\ 
%		$10^{-11}$ & 2.18e-05 & 1.09e-05 & 3.38e-07 & 8.41e-08 & 1.89e-09 & 2.37e-10 & 1.91e-11 & 1.19e-12 \\ 
		$10^{-12}$ & 1.33e-05 & 6.66e-06 & 3.53e-07 & 8.79e-08 & 1.38e-09 & 1.72e-10 & 1.80e-11 & 1.12e-12 \\ 
%		$10^{-13}$ & 8.12e-06 & 4.05e-06 & 3.94e-07 & 9.83e-08 & 9.81e-10 & 1.23e-10 & 1.92e-11 & 1.20e-12 \\ 
		$10^{-14}$ & 4.95e-06 & 2.45e-06 & 4.49e-07 & 1.12e-07 & 6.87e-10 & 8.58e-11 & 2.31e-11 & 1.44e-12 \\ 
		\bottomrule
	\end{tabular}
	\caption{Energy norm error for finite elements of order $k = 1, \ldots, 4$ applied
	to Example~\ref{test:SunStynes} with certain $\eps$ and $\lambda = 0.005$.}
	\label{tab:energy}
	\end{center}
	\end{table}
		
	Furthermore, we study the influence of varying $\lambda$ and $\alpha_0$. Therefore,
	we consider Example~\ref{test:SunStynes} with fixed $\eps = 10^{-8}$ on
	correspondent layer-adapted meshes with $N=1024$. In Figure~\ref{fig:lambda}
	the energy norm error is plotted against $\lambda$ for $\alpha_0 = 1$ (left) and
	against $\alpha_0$ for $\lambda = 0.005$ (right), respectively. In both cases the
	error is almost constant in the studied ranges. Thus, it seems to be plausible
	to presume the method to be robust in $\alpha$.
	
	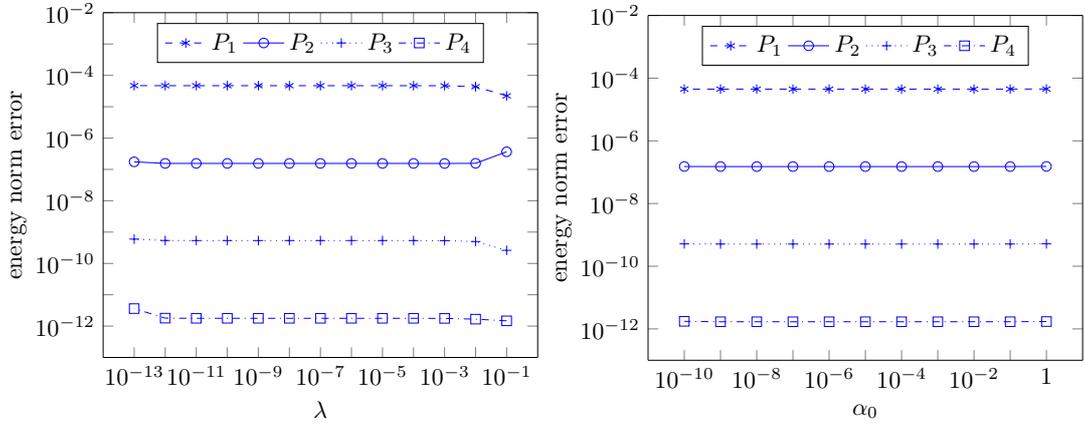
\begin{figure}
		\begin{center}
% energy norm error for the testproblem of Sun and Stynes
% with epsilon = 10^{-8} and N = 1024
% lambda varies between 10^{-13} and 10^{-1}
%
%
\begin{tikzpicture}[scale = 0.9]

\begin{axis}[%
width=2.5in,
height=2in,
at={(0.809in,0.513in)},
scale only axis,
separate axis lines,
every outer x axis line/.append style={black},
every x tick label/.append style={font=\color{black}},
xmode=log,
xmin=1e-14,
xmax=1,
xtick={ 1e-13,  1e-12,  1e-11,  1e-10,  1e-09,  1e-08,  1e-07,  1e-06,  1e-05, 0.0001,  0.001,   0.01,    0.1},
xticklabels={ $10^{-13}$,  ,  $10^{-11}$, ,  $10^{-9}$,  ,  $10^{-7}$,  ,  $10^{-5}$, ,  $10^{-3}$,   , $10^{-1}$},
xminorticks=false,
xlabel={$\lambda$},
every outer y axis line/.append style={black},
every y tick label/.append style={font=\color{black}},
ymode=log,
ymin=1e-13,
ymax=0.01,
ytick={ 1e-12,  1e-11,  1e-10,  1e-09,  1e-08,  1e-07,  1e-06,  1e-05, 0.0001,  0.001,   0.01},
yticklabels={ $10^{-12}$,  ,  $10^{-10}$, ,  $10^{-8}$,  ,  $10^{-6}$,  ,  $10^{-4}$, ,  $10^{-2}$},
yminorticks=false,
ylabel={energy norm error},
axis background/.style={fill=white},
legend style={at={(0.5,0.97)},anchor=north,legend columns=4,legend cell align=left,align=left,draw=black}
]
\addplot [color=blue,dashed,mark=asterisk,mark options={solid}]
  table[row sep=crcr]{%
0.1	2.22043067303763e-05\\
0.01	4.35077014376994e-05\\
0.001	4.65778727128461e-05\\
0.0001	4.68969808908077e-05\\
1e-05	4.69290158899527e-05\\
1e-06	4.69322206357255e-05\\
1e-07	4.69325411281134e-05\\
1e-08	4.69325732300405e-05\\
1e-09	4.69325757124149e-05\\
1e-10	4.69325954196188e-05\\
1e-11	4.69325493984098e-05\\
1e-12	4.69439035146029e-05\\
1e-13	4.71552506077022e-05\\
};
\addlegendentry{$P_1$};

\addplot [color=blue,solid,mark=o,mark options={solid}]
  table[row sep=crcr]{%
0.1	3.645075541586e-07\\
0.01	1.56334723100479e-07\\
0.001	1.54116101418501e-07\\
0.0001	1.55004967995591e-07\\
1e-05	1.55106011068979e-07\\
1e-06	1.55116237519538e-07\\
1e-07	1.55117261307625e-07\\
1e-08	1.55117364464971e-07\\
1e-09	1.55117349532449e-07\\
1e-10	1.5511761028899e-07\\
1e-11	1.55121076375886e-07\\
1e-12	1.55188600677303e-07\\
1e-13	1.7405905665404e-07\\
};
\addlegendentry{$P_2$};

\addplot [color=blue,dotted,mark=+,mark options={solid}]
  table[row sep=crcr]{%
0.1	2.61496804545942e-10\\
0.01	4.98939055377181e-10\\
0.001	5.32425022351273e-10\\
0.0001	5.35896708454065e-10\\
1e-05	5.3624513868085e-10\\
1e-06	5.36279994377855e-10\\
1e-07	5.36283475480354e-10\\
1e-08	5.36283831181367e-10\\
1e-09	5.36284126699955e-10\\
1e-10	5.3628340527857e-10\\
1e-11	5.36307780863329e-10\\
1e-12	5.3877020390162e-10\\
1e-13	5.96359743196154e-10\\
};
\addlegendentry{$P_3$};

\addplot [color=blue,dashdotted,mark=square,mark options={solid}]
  table[row sep=crcr]{%
0.1	1.4725891828844e-12\\
0.01	1.66659044817231e-12\\
0.001	1.76228878538549e-12\\
0.0001	1.77358174347547e-12\\
1e-05	1.77472828936459e-12\\
1e-06	1.77484296667645e-12\\
1e-07	1.77485457545357e-12\\
1e-08	1.77485546591019e-12\\
1e-09	1.77485562918018e-12\\
1e-10	1.77485715101978e-12\\
1e-11	1.77489533945522e-12\\
1e-12	1.79491159899424e-12\\
1e-13	3.61815730468383e-12\\
};
\addlegendentry{$P_4$};

\end{axis}
\end{tikzpicture}%
% energy norm error for the testproblem of Sun and Stynes
% with lambda = 0.005, epsilon = 10^{-8} and N = 1024
% alpha0 varies between 10^{-10} and 1
%
%
\begin{tikzpicture}[scale = 0.9]

\begin{axis}[%
width=2.5in, %4.822in,
height=2in,  %3.794in,
at={(0.809in,0.522in)},
scale only axis,
separate axis lines,
every outer x axis line/.append style={black},
every x tick label/.append style={font=\color{black}},
xmode=log,
xmin=1e-11,
xmax=10,
xtick={ 1e-10,  1e-09,  1e-08,  1e-07,  1e-06,  1e-05, 0.0001,  0.001,   0.01,    0.1, 1},
xticklabels={ $10^{-10}$,  ,  $10^{-8}$, ,  $10^{-6}$,  ,  $10^{-4}$,  ,  $10^{-2}$, ,  $1$},
xminorticks=false,
xlabel={$\alpha_0$},
every outer y axis line/.append style={black},
every y tick label/.append style={font=\color{black}},
ymode=log,
ymin=1e-13,
ymax=0.01,
yminorticks=false,
ylabel={energy norm error},
axis background/.style={fill=white},
legend style={at={(0.5,0.97)},anchor=north,legend columns=4,legend cell align=left,align=left,draw=black}
]
\addplot [color=blue,dashed,mark=asterisk,mark options={solid}]
  table[row sep=crcr]{%
1	4.5186568165753e-05\\
0.1	4.47796438601279e-05\\
0.01	4.4739243021703e-05\\
0.001	4.47352058383787e-05\\
0.0001	4.4734802146987e-05\\
1e-05	4.47347617655925e-05\\
1e-06	4.47347578772116e-05\\
1e-07	4.47347599279568e-05\\
1e-08	4.47347525408896e-05\\
1e-09	4.47351397868374e-05\\
1e-10	4.47666624323351e-05\\
};
\addlegendentry{$P_1$};

\addplot [color=blue,solid,mark=o,mark options={solid}]
  table[row sep=crcr]{%
1	1.52746537638214e-07\\
0.1	1.51090800711023e-07\\
0.01	1.50926843083497e-07\\
0.001	1.50910463417783e-07\\
0.0001	1.50908825235442e-07\\
1e-05	1.50908661973134e-07\\
1e-06	1.50908650054335e-07\\
1e-07	1.50908615688729e-07\\
1e-08	1.50909037583026e-07\\
1e-09	1.50907178717934e-07\\
1e-10	1.51843513803633e-07\\
};
\addlegendentry{$P_2$};

\addplot [color=blue,dotted,mark=+,mark options={solid}]
  table[row sep=crcr]{%
1	5.17269423841341e-10\\
0.1	5.10375333830083e-10\\
0.01	5.09691834249559e-10\\
0.001	5.09623543142384e-10\\
0.0001	5.09616714492954e-10\\
1e-05	5.09616028569591e-10\\
1e-06	5.09615964910052e-10\\
1e-07	5.09614017156339e-10\\
1e-08	5.09618588053653e-10\\
1e-09	5.09741056977722e-10\\
1e-10	5.15906559859399e-10\\
};
\addlegendentry{$P_3$};

\addplot [color=blue,dashdotted,mark=square,mark options={solid}]
  table[row sep=crcr]{%
1	1.71604482089942e-12\\
0.1	1.69197268758834e-12\\
0.01	1.68954898049942e-12\\
0.001	1.68934470686527e-12\\
0.0001	1.68932680259355e-12\\
1e-05	1.68933241819242e-12\\
1e-06	1.6892985842038e-12\\
1e-07	1.68930129224902e-12\\
1e-08	1.68934803004146e-12\\
1e-09	1.68865180395923e-12\\
1e-10	1.74268122114982e-12\\
};
\addlegendentry{$P_4$};

\end{axis}
\end{tikzpicture}%
			\caption{Energy norm error for $N = 1024$ and finite elements of order
			$k = 1, \ldots, 4$ applied to Example~\ref{test:SunStynes} with
			$\eps = 10^{-8}$, $\lambda = 10^{-13},\ldots,10^{-1}$, and $\alpha_0=1$ (left).
			Same setting with $\lambda = 0.005$ and $\alpha_0=10^{-10},\ldots,1$ (right).}
			\label{fig:lambda}
		\end{center}
	\end{figure}
	
	Finally, in Table~\ref{tab:l2norm} we compare the energy norm and the $L^2$-norm
	error for linear finite elements. As predicted by theory, cf.
	Theorem~\ref{th:L2ErrorLiseikin}, the $L^2$-error is uniformly convergent of
	second order whereas the error in the $\tnorm{\cdot}_\eps$-norm	converges
	with order one only.
	
	In summary, our numerical experiments confirm the theoretical results of
	Section~\ref{sec:FEM-analysis}.
	
	\begin{table}
	\begin{center}
	\begin{tabular}{ r *{8}{c} } \toprule
			& \multicolumn{4}{c}{$\eps=10^{-8}$} & \multicolumn{4}{c}{$\eps=10^{-12}$} \\
		\cmidrule(l{0.7ex}r{0.7ex}){2-5} \cmidrule(l{0.7ex}r{0.7ex}){6-9}
			& \multicolumn{2}{c}{$\tnorm{u-u_N}_\eps$} & \multicolumn{2}{c}{$\norm{u-u_N}$} &
				\multicolumn{2}{c}{$\tnorm{u-u_N}_\eps$} & \multicolumn{2}{c}{$\norm{u-u_N}$} \\
		$N$ & error & rates & error & rates & error & rates & error & rates \\ 
		\cmidrule(l{0.7ex}r{0.7ex}){1-1}
		\cmidrule(l{0.7ex}r{0.7ex}){2-3} \cmidrule(l{0.7ex}r{0.7ex}){4-5}
		\cmidrule(l{0.7ex}r{0.7ex}){6-7} \cmidrule(l{0.7ex}r{0.7ex}){8-9}
		   8 & 7.58e-03 & 1.402    & 4.11e-03 & 2.467    & 2.22e-03 & 1.623    & 2.06e-03 & 1.846    \\ 
		  16 & 2.87e-03 & 0.986    & 7.43e-04 & 2.108    & 7.22e-04 & 1.460    & 5.72e-04 & 1.853    \\ 
		  32 & 1.45e-03 & 1.002    & 1.72e-04 & 2.009    & 2.62e-04 & 1.203    & 1.59e-04 & 1.947    \\ 
		  64 & 7.23e-04 & 1.001    & 4.28e-05 & 2.002    & 1.14e-04 & 1.070    & 4.11e-05 & 1.986    \\ 
		 128 & 3.62e-04 & 1.000    & 1.07e-05 & 2.000    & 5.43e-05 & 1.019    & 1.04e-05 & 1.997    \\ 
		 256 & 1.81e-04 & 1.000    & 2.67e-06 & 2.000    & 2.68e-05 & 1.005    & 2.60e-06 & 1.999    \\ 
		 512 & 9.04e-05 & 1.000    & 6.68e-07 & 2.000    & 1.33e-05 & 1.001    & 6.50e-07 & 2.000    \\ 
		1024 & 4.52e-05 & 1.000    & 1.67e-07 & 2.000    & 6.66e-06 & 1.000    & 1.63e-07 & 2.000    \\ 
		2048 & 2.26e-05 & 1.000    & 4.17e-08 & 2.000    & 3.33e-06 & 1.000    & 4.07e-08 & 2.000    \\ 
		\midrule
		theory & & 1 & & 2 & & 1 & & 2 \\
		\bottomrule
	\end{tabular}
	\caption{Energy norm  and $L^2$-norm error for linear finite elements applied
	to Example~\ref{test:SunStynes} with $\eps = 10^{-8}, 10^{-12}$ and $\lambda = 0.005$.}
	\label{tab:l2norm}
	\end{center}
	\end{table}
	
	\section*{Acknowledgement}
	The author would like to thank Hans-G{\"o}rg Roos for helpful comments and discussions.
	
  %\newpage
  \begin{appendix}
	
	\section{Auxiliary lemmas}
	
	In this section we provide some auxiliary lemmas and prove some basic inequalities
	that are needed in the paper.
	\begin{lemma}
		\label{le:sumToSum}
		Let $\alpha \in \ZBR, \, 0< \alpha \leq 1$ and let $a, b > 0$. Then
		\begin{gather*}
			\begin{aligned}
				\left(a+b\right)^{1/\alpha} &\leq 2^{1/\alpha-1} \left(a^{1/\alpha} + b^{1/\alpha}\right)\!, \\
				\left(a+b\right)^\alpha &\leq a^\alpha + b^\alpha, 
			\end{aligned} \qquad \text{for } 0 < \alpha \leq 1.
		\end{gather*}
	\end{lemma}
	
	\begin{proof}
		To prove the first inequality, we use that $x \mapsto x^{1/\alpha}$ is convex for $0<\alpha \leq 1$. Hence,
		\begin{gather*}
			\left(a+b\right)^{1/\alpha} = 2^{1/\alpha} \left(\frac{1}{2} a + \frac{1}{2} b\right)^{\!1/\alpha}
				\leq 2^{1/\alpha} \left(\frac{1}{2} a^{1/\alpha} + \frac{1}{2} b^{1/\alpha} \right)
				= 2^{1/\alpha -1} \left(a^{1/\alpha} + b^{1/\alpha} \right)\!.
		\end{gather*}
		We gain the second inequality by studying the function $f: (0,\infty) \to \ZBR$ which is defined by
		$f(x) = (a+x)^\alpha - (a^\alpha + x^\alpha)$. Since $a>0$, $0<\alpha \leq 1$, and $x \mapsto x^{\alpha-1}$
		is monotonically decreasing we have
		\begin{gather*}
			f'(x) = \alpha \big((a+x)^{\alpha-1}-x^{\alpha-1}\big) \leq 0 \qquad \text{for } x > 0.
		\end{gather*}
		Hence, also $f$ is monotonically decreasing and for $x \in (0,b)$
		\begin{gather*}
			(a+b)^\alpha - (a^\alpha +b^\alpha) \leq f(x) \leq \lim_{x\to 0}f(x) = (a+0)^\alpha-(a^\alpha+0^\alpha) = a^\alpha - a^\alpha = 0.
		\end{gather*}
	\end{proof}
	
	\begin{lemma}
		\label{le:1pcmc}
		Let $\alpha, c \in [0,1]$. Then $2^\alpha -1 \leq (1+c)^\alpha - c^\alpha \leq 1$.
	\end{lemma}
	\begin{proof}
		Since the case $\alpha = 0$ is easy, we assume $\alpha > 0$.
		The proof uses the monotonicity properties of the function $x \mapsto (1+x)^\alpha - x^\alpha$.
		To detect these properties we study the first derivative
		\begin{gather*}
			\frac{\partial}{\partial x} \left[(1+x)^\alpha - x^\alpha\right]
				= \alpha \left[(1+x)^{\alpha-1} - x^{\alpha-1}\right]
		\end{gather*}
		for $x \in (0,1)$. By assumption $\alpha - 1 \leq 0$. Consequently, $(1+x)^{\alpha-1}\leq 1$
		and $x^{\alpha -1} \geq 1$ for $x \in (0,1)$. Therefore, the first derivative of
		$x \mapsto (1+x)^\alpha - x^\alpha$ is negative in $(0,1)$ and the function is monotonically
		decreasing in this interval. Hence,
		\begin{gather*}
			2^\alpha-1 = (1+1)^\alpha - 1^\alpha \leq (1+c)^\alpha - c^\alpha \leq (1+0)^\alpha - 0^\alpha = 1.
		\end{gather*}
	\end{proof}
	
	\begin{lemma}
		\label{le:1pcmcII}
		Let $\alpha, c \in (0,1]$. Then
		\begin{gather*}
			0 < (1+c)^\alpha - c^\alpha
			  \leq \frac{\big(2^\alpha - 1\big)}{\ln(2)} \big(\ln (1+ c) - \ln (c)\big).
		\end{gather*}
		Furthermore, we have
		\begin{gather*}
			\alpha \ln(2) \leq \big(2^\alpha -1\big) \leq \alpha.
		 \end{gather*}
	\end{lemma}
	\begin{proof}
		We study the function
		\begin{gather*}
			f(x) = \frac{\alpha \ln (1+x) - \alpha \ln(x)}{(1+x)^\alpha - x^\alpha}
				= \frac{\ln\! \left(\left(\frac{1+x}{x}\right)^\alpha\right)}{(1+x)^\alpha - x^\alpha}
		\end{gather*}
		for $x \in (0,1], \, \alpha \in (0,1]$. The first derivative of $f$ is calculated to be
		\begin{align*}
			f'(x)
				& = \frac{\alpha\left(\frac{1}{1+x}-\frac{1}{x}\right)\big((1+x)^\alpha-x^\alpha\big)
					- \alpha \ln\! \left(\left(\frac{1+x}{x}\right)^\alpha\right)\big((1+x)^{\alpha-1}-x^{\alpha-1}\big)
					}{\big((1+x)^\alpha-x^\alpha\big)^2} \\
				& = \frac{\alpha\left(x-(1+x)\right)\big((1+x)^\alpha-x^\alpha\big)
					- \alpha \ln\! \left(\left(\frac{1+x}{x}\right)^\alpha\right)\big((1+x)^{\alpha-1}-x^{\alpha-1}\big)x(1+x)
					}{x(1+x)\big((1+x)^\alpha-x^\alpha\big)^2} \\
				& = \frac{-\alpha}{x(1+x)\big((1+x)^\alpha-x^\alpha\big)^2}
					\underbrace{\Big[\big((1+x)^\alpha-x^\alpha\big)
					+ \ln\! \left(\left(\tfrac{1+x}{x}\right)^\alpha\right)\big(x(1+x)^\alpha-x^\alpha (1+x)\big)\Big]}_{=:g(x)}.
		\end{align*}
		By Lemma~\ref{le:sumToSum} we have for $x \in (0,1], \, \alpha \in (0,1]$
		\begin{gather*}
			x(1+x)^\alpha-x^\alpha(1+x)
				\leq x(1^\alpha + x^\alpha)-x^\alpha-x^{1+\alpha}
				= x - x^\alpha
				\leq 0.
		\end{gather*}
		Therefore, due to $\ln(x) \leq x-1$ we obtain
		\begin{align*}
			g(x)
				& \geq \big((1+x)^\alpha-x^\alpha\big)
					+ \big(\left(\tfrac{1+x}{x}\right)^\alpha-1\big)\big(x(1+x)^\alpha-x^\alpha (1+x)\big) \\
				& = \big((1+x)-1\big)x^\alpha +\big(1-x-(1+x)\big)(1+x)^\alpha+ x^{1-\alpha}(1+x)^{2\alpha} \\
				& = x^{1+\alpha} - 2x(1+x)^\alpha + x^{1-\alpha}(1+x)^{2\alpha} \\
				& = \left(x^{(1+\alpha)/2} - x^{(1-\alpha)/2}(1+x)^\alpha\right)^2 \\
				& \geq 0.
		\end{align*}
		Thus, $f'(x) \leq 0$ for $x\in (0,1], \, \alpha \in (0,1]$ and $f$ is monotonically decreasing. This yields
		\begin{gather*}
			f(x)= \frac{\alpha \ln (1+x) - \alpha \ln(x)}{(1+x)^\alpha - x^\alpha}
				\geq \frac{\alpha \ln (2) - \alpha \ln(1)}{2^\alpha - 1^\alpha}
				= \frac{\alpha \ln(2)}{2^\alpha-1}
		\end{gather*}
		and for $x = c \in (0,1]$
		\begin{gather*}
			(1+c)^\alpha - c^\alpha \leq \frac{\big(2^\alpha-1\big)}{\ln(2)} \big(\ln(1+c)-\ln(c)\big).
		\end{gather*}
		The lower bound of the second statement follows from $1+x \leq e^x$ since
		\begin{gather*}
			1+ \alpha \ln(2) \leq e^{\alpha \ln(2)} = 2^\alpha.
		\end{gather*}
		The upper bound is verified using the convexity of $x \mapsto 2^x$. We have
		\begin{gather*}
			2^\alpha = 2^{\alpha \cdot 1 +(1-\alpha) \cdot 0} \leq \alpha \cdot 2^1 + (1-\alpha) \cdot 2^0 = 2\alpha + (1-\alpha) = 1+\alpha.
		\end{gather*}
	\end{proof}
	\begin{remark}
		The bound of the last lemma is exact for $c = 1$ and asymptotically exact for $\alpha \searrow 0$.
	\end{remark}
	
	\newpage
	\section{Proof of Lemmas~\ref{le:liseikinhBounds} and~\ref{le:liseikinh^kBounds}}
	\label{app:proofLiseikin}
	
	For $\xi \geq 0$ we have
	\begin{subequations}
	\label{eq:phiDer}
	\begin{align}
		\phj(\xi,\eps)
			& =\left(\eps^{\alpha/2} + \xi \left[(1+\eps^{1/2})^\alpha-\eps^{\alpha/2}\right]\right)^{1/\alpha} - \eps^{1/2}, \\
		\frac{\partial}{\partial \xi} \phj(\xi, \eps)
			& = \frac{1}{\alpha} \left[(1+\eps^{1/2})^\alpha-\eps^{\alpha/2}\right]
				\left(\eps^{\alpha/2} + \xi \left[(1+\eps^{1/2})^\alpha-\eps^{\alpha/2}\right]\right)^{(1-\alpha)/\alpha}, \\
		\frac{\partial^2}{\partial \xi^2} \phj(\xi, \eps)
			& = \frac{1-\alpha}{\alpha^2} \left[(1+\eps^{1/2})^\alpha-\eps^{\alpha/2}\right]^2
				\left(\eps^{\alpha/2} + \xi \left[(1+\eps^{1/2})^\alpha-\eps^{\alpha/2}\right]\right)^{(1-2\alpha)/\alpha}.
	\end{align}
	\end{subequations}
	Using the definition of the mesh points $x_i = \phj(ih,\eps)$ we obtain
	\begin{gather}
		\label{eq:xiPsqrteps}
		x_i + \eps^{1/2}
			 = \left(\eps^{\alpha/2} + i h \left[(1+\eps^{1/2})^\alpha-\eps^{\alpha/2}\right]\right)^{1/\alpha}.
	\end{gather}
	
	\begin{proof}[of Lemma~\ref{le:liseikinh^kBounds}]
		Recalling~\eqref{eq:phiDer} and~\eqref{eq:xiPsqrteps}, using the mean value theorem,
		and the monotony of $\frac{\partial}{\partial \xi} \phj(\xi, \eps)$
		(here $\alpha \leq 1$ is needed) we can bound the lengths of the mesh intervals as follows
		\begin{gather}
		\label{eq:hibound}
		\begin{aligned}
			h_i & = x_i - x_{i-1} = \phj\!\left(ih,\eps\right) - \phj\!\left((i-1)h,\eps\right) \\
				& \leq h \, \frac{\partial \phj}{\partial \xi}\!\left(ih,\eps\right) \\
				& = h \frac{1}{\alpha} \left[(1+\eps^{1/2})^\alpha-\eps^{\alpha/2}\right]
					\left(\eps^{\alpha/2} + i h \left[(1+\eps^{1/2})^\alpha-\eps^{\alpha/2}\right]\right)^{(1-\alpha)/\alpha} \\
				& = h \frac{1}{\alpha} \left[(1+\eps^{1/2})^\alpha-\eps^{\alpha/2}\right]
					\left(x_i + \eps^{1/2}\right)^{(1-\alpha)}.
		\end{aligned}
		\end{gather}
		Consequently, for $0<\alpha \leq 1$ we get the third wanted estimate
		\begin{gather*}
			h_i \leq h \frac{1}{\alpha} \left[(1+\eps^{1/2})^\alpha-\eps^{\alpha/2}\right]
					\left(x_i + \eps^{1/2}\right)^{(1-\alpha)}
				\leq h \kappa 2^{(1-\alpha)}
				\leq 2 \kappa h \leq C h.
		\end{gather*}
		
		Now, let $0 < \alpha \leq 1/k$ with $k \in \ZBN,\, k\geq 1$ and $\eps \leq h^{2/\alpha}$. Using the
		definition of the mesh~\eqref{eq:liseikinMeshGenerator} and Lemma~\ref{le:1pcmc} we obtain
		\begin{align*}
			x_1 & = \left(\eps^{\alpha/2} + h \left[(1+\eps^{1/2})^\alpha-\eps^{\alpha/2}\right]\right)^{1/\alpha} - \eps^{1/2} \\
				& \leq \left(h + h \left[(1+\eps^{1/2})^\alpha-\eps^{\alpha/2}\right]\right)^{1/\alpha} \\
				& \leq \left(2 h\right)^{1/\alpha}
					= \left(2 h^{(1-k\alpha)}\right)^{1/\alpha} h^k \\
				& \leq 2^{1/\alpha} h^k.
		\end{align*}
		Note that for $0<\alpha\leq 1/(2k)$ and $h\leq 1/4$ an $\alpha$-independent estimate
		is guaranteed because of
		\begin{gather*}
			2 h^{(1-k\alpha)} \leq 2 h^{1/2} \leq 2 \left(\tfrac{1}{4}\right)^{1/2} = 1.
		\end{gather*}
		
		Finally, let $\hat{\alpha} > 0$ and $0< \alpha \leq \min\{\hat{\alpha}/k,1\}$ with
		$k \in \ZBN,\, k\geq 1$.
		If $\eps \geq h^{2/\alpha}$ we obtain by~\eqref{eq:hibound} and Lemma~\ref{le:1pcmc}
		\begin{align*}
			h_1 \left( x_{0}+\eps^{1/2} \right)^{\hat{\alpha}/k-1}
				& = h_1 \eps^{(\hat{\alpha}/k-1)/2} \\
				& \leq h \frac{1}{\alpha} \left[(1+\eps^{1/2})^\alpha-\eps^{\alpha/2}\right]
					\left(\eps^{\alpha/2} + h \left[(1+\eps^{1/2})^\alpha-\eps^{\alpha/2}\right]\right)^{(1-\alpha)/\alpha}
					\eps^{(\hat{\alpha}/k-1)/2} \\
				& \leq h \kappa \left(2\eps^{\alpha/2}\right)^{(1-\alpha)/\alpha}
					\eps^{(\hat{\alpha}/k-1)/2} \\
				& = h \kappa 2^{(1-\alpha)/\alpha} \eps^{(\hat{\alpha}/k-\alpha)/2} \\
				& \leq C h.
		\end{align*}
		In general, the estimate~\eqref{eq:hibound} and the identity~\eqref{eq:xiPsqrteps} yield
		for $1\leq i \leq N$
		\begin{align*}
			& h_i \left( x_{i-1}+\eps^{1/2} \right)^{\hat{\alpha}/k-1}\! \\
			& \quad \leq h \frac{1}{\alpha} \left[(1+\eps^{1/2})^\alpha-\eps^{\alpha/2}\right]
				\left(x_i + \eps^{1/2}\right)^{(1-\alpha)} \left( x_{i-1}+\eps^{1/2} \right)^{\hat{\alpha}/k-1} \\
			& \quad = h \frac{1}{\alpha} \left[(1+\eps^{1/2})^\alpha-\eps^{\alpha/2}\right]
				\left(\frac{x_i + \eps^{1/2}}{x_{i-1}+\eps^{1/2}}\right)^{\!(1-\alpha)}
				\left( x_{i-1}+\eps^{1/2} \right)^{\hat{\alpha}/k-\alpha} \\
			& \quad = h \frac{1}{\alpha} \left[(1+\eps^{1/2})^\alpha-\eps^{\alpha/2}\right]
				\left(\frac{\eps^{\alpha/2} + i h \left[(1+\eps^{1/2})^\alpha-\eps^{\alpha/2}\right]}{
					\eps^{\alpha/2} + (i-1) h \left[(1+\eps^{1/2})^\alpha-\eps^{\alpha/2}\right]}\right)^{\!(1-\alpha)/\alpha}
				\!\!\left( x_{i-1}+\eps^{1/2} \right)^{\hat{\alpha}/k-\alpha}\!.
		\end{align*}
		Furthermore, for $2\leq i\leq N$ we have
		\begin{align*}
			\frac{\eps^{\alpha/2} + i h \left[(1+\eps^{1/2})^\alpha-\eps^{\alpha/2}\right]}{
					\eps^{\alpha/2} + (i-1) h \left[(1+\eps^{1/2})^\alpha-\eps^{\alpha/2}\right]}
				& = 1 + \frac{h \left[(1+\eps^{1/2})^\alpha-\eps^{\alpha/2}\right]}{
					\eps^{\alpha/2} + (i-1) h \left[(1+\eps^{1/2})^\alpha-\eps^{\alpha/2}\right]} \\
				& \leq 1 + \frac{h \left[(1+\eps^{1/2})^\alpha-\eps^{\alpha/2}\right]}{
					(i-1) h \left[(1+\eps^{1/2})^\alpha-\eps^{\alpha/2}\right]} \\
				& = 1 + \frac{1}{i-1} = \frac{i}{i-1}.
		\end{align*}
		Hence, %with~\eqref{kappa}
		we get
		\begin{align*}
			h_i \left( x_{i-1}+\eps^{1/2} \right)^{\hat{\alpha}/k-1}\!
				& \leq h \kappa \left(\frac{i}{i-1}\right)^{\!(1-\alpha)/\alpha}
					\!\!\left( x_{i-1}+\eps^{1/2} \right)^{\hat{\alpha}/k-\alpha} \\
				& \leq C h	\qquad \qquad \text{ for } 2 \leq i \leq N.
		\end{align*}
		Raising the inequalities to the $k$-th power gives the wanted estimate.
	\end{proof}
	
	\begin{proof}[of Lemma~\ref{le:liseikinhBounds}]
		First of all, recall~\eqref{eq:phiDer}. For $i = 2,\ldots,N$ the mean value
		theorem, the monotony of $\frac{\partial}{\partial \xi} \phj(\xi, \eps)$
		(here $0< \alpha \leq \tfrac{1}{2}$ is needed), \eqref{eq:xiPsqrteps},
		and~\eqref{kappa} enable to bound the difference $h_i - h_{i-1}$ as follows
		\begin{align*}
			h_i-h_{i-1}
				& = \Big( \phj\big(ih,\eps\big) - \phj\big((i-1)h,\eps\big)\Big)
					- \Big( \phj\big((i-1)h,\eps\big) - \phj\big((i-2)h,\eps\big) \Big) \\
				& = h \left(\frac{\partial \phj}{\partial \xi}\!\left(\xi_i ,\eps\right)
					-\frac{\partial \phj}{\partial \xi}\!\left(\xi_{i-2} ,\eps\right) \right) \\
				& = h \left(\xi_i - \xi_{i-2}\right) \frac{\partial^2 \phj}{\partial \xi^2}\!\left(\xi_{i-1},\eps\right) \\
				& \leq 2 h^2 \frac{\partial^2 \phj}{\partial \xi^2}\!\left(ih,\eps\right) \\
				& = 2 h^2 \frac{1-\alpha}{\alpha^2} \left[(1+\eps^{1/2})^\alpha-\eps^{\alpha/2}\right]^2
					\left(\eps^{\alpha/2} + ih \left[(1+\eps^{1/2})^\alpha-\eps^{\alpha/2}\right]\right)^{(1-2\alpha)/\alpha} \\
				& = 2h^2 (1-\alpha)\kappa^2 \left( x_{i} + \eps^{1/2} \right)^{1- 2\alpha}
					\leq C h^2 \left( x_{i} + \eps^{1/2} \right)^{1- 2\alpha}
		\end{align*}
		where $\xi_{i} \in \left((i-1)h,ih\right)$,
		$\xi_{i-2} \in \left((i-2)h,(i-1) h\right)$, and
		$\xi_{i-1} \in \left(\xi_{i-2},\xi_{i}\right)$.
		
		Using this estimate and applying the same technique as in the proof
		of Lemma~\ref{le:liseikinh^kBounds}, we obtain
		\begin{align*}
			\left( h_i-h_{i-1} \right) \left( x_{i-1} + \eps^{1/2} \right)^{\hat{\alpha} - 1}
				& \leq C h^2 \left(\frac{i}{i-1}\right)^{\!(1-2\alpha)/\alpha}
					\!\!\left( x_{i-1}+\eps^{1/2} \right)^{\hat{\alpha}-2\alpha}
				\leq C h^2
		\end{align*}
		for $2 \leq i \leq N$.
	\end{proof}
	
	\section{Proof of Lemma~\ref{le:normEquiVN}}
	\label{app:proofNorm}
	
	\begin{proof}[of Lemma~\ref{le:normEquiVN}]
		An easy calculation shows
		\begin{align*}
			\norm{e}_{(x_L,x_R)}
				& = \left( \sum_{i=L+1}^R \int_{x_{i-1}}^{x_i} \left[ e_i- \frac{x_i-x}{h_i} (e_i-e_{i-1})\right]^2 dx \right)^{\!1/2} \\
				& = \left( \sum_{i=L+1}^R \int_{x_{i-1}}^{x_i} e_i^2
					- 2 \frac{x_i-x}{h_i} (e_i-e_{i-1})e_i + \frac{(x_i-x)^2}{h_i^2} (e_i-e_{i-1})^2 dx \right)^{\!1/2} \\
				& = \left( \sum_{i=L+1}^R \left( e_i^2 h_i - \frac{(x_i-x_{i-1})^2}{h_i} (e_i-e_{i-1})e_i
					+ \frac{1}{3}\frac{(x_i-x_{i-1})^3}{h_i^2} (e_i-e_{i-1})^2 \right) \right)^{\!1/2} \\
				& = \left(\sum_{i=L+1}^R h_i \left( e_i^2 - e_i^2 + e_i e_{i-1} 
					+ \frac{1}{3} e_i^2 - \frac{2}{3} e_i e_{i-1} + \frac{1}{3} e_{i-1}^2 \right) \right)^{\!1/2} \\
				& = \left(\frac{1}{3} \sum_{i=L+1}^R h_i \left( e_i^2 + e_i e_{i-1} + e_{i-1}^2 \right) \right)^{\!1/2}
					\geq \left(\frac{1}{6} \sum_{i=L+1}^R h_i \left( e_i^2 + e_{i-1}^2 \right) \right)^{\!1/2}.
		\end{align*}
		Furthermore, using the Cauchy-Schwarz inequality, we get
		\begin{align*}
			& \sum_{i=L+1}^{R-1} \hbar_i |e_i| + \frac{1}{2} \left(h_{L+1} |e_L| + h_{R}|e_R| \right) \\
			& \quad = \frac{1}{2} \sum_{i=L+1}^{R-1} \left( h_i |e_i| + h_{i+1} |e_i| \right)
					+ \frac{1}{2} \left(h_{L+1} |e_L| + h_{R}|e_R| \right)
				= \frac{1}{2} \sum_{i=L+1}^{R} h_i \left( |e_i| + |e_{i-1}| \right) \\
			& \quad \leq \left(\frac{1}{2} \sum_{i=L+1}^{R} h_i \right)^{\!1/2}
					\left(\frac{1}{2} \sum_{i=L+1}^{R} h_i \left( |e_i| + |e_{i-1}| \right)^2 \right)^{\!1/2}
%				& \leq \left( \frac{x_R-x_{L}}{2}\right)^{\!1/2}
%					\!\left(\frac{1}{2} \sum_{i=L+1}^{R} h_i \left( e_i^2 + 2 |e_i| |e_{i-1}| + e_{i-1}^2 \right) \right)^{\!1/2} \\
			\leq \left( \frac{x_R-x_{L}}{2}\right)^{\!1/2}
					\!\left( \sum_{i=L+1}^{R} h_i \left( e_i^2 + e_{i-1}^2 \right) \right)^{\!1/2}.
		\end{align*}
		Hence, we have
		\begin{gather*}
			\sum_{i=L+1}^{R-1} \hbar_i |e_i| + \frac{1}{2} \left(h_{L+1} |e_L| + h_{R}|e_R| \right)
				\leq \sqrt{3} \left( x_R-x_{L} \right)^{1/2} \norm{e}_{(x_L,x_R)}\!.
		\end{gather*}
	\end{proof}	
	
  \end{appendix}

	\bibliographystyle{plain}
	\bibliography{turningPointBiblio}
	
\end{document}